\documentclass[12pt]{amsart}

\usepackage{mathabx}

\usepackage{amsmath,accents}

\usepackage{geometry}
\usepackage{enumitem}
\setlist[itemize]{leftmargin=*,  label=\rm{(\arabic*)}}
\setlist[enumerate]{leftmargin=*, label=\rm{(\arabic*)}}

\newcommand\nl{\par\noindent}
\numberwithin{equation}{subsection}

\newtheorem*{thm*}{Theorem}
\newtheorem{thm}{Theorem}[section]
\newtheorem{pro}[thm]{Proposition}
\newtheorem{lem}[thm]{Lemma}
\newtheorem{cor}[thm]{Corollary}
\newtheorem*{cor*}{Corollary}

\theoremstyle{definition}\newtheorem{defi}[thm]{Definition}
\theoremstyle{definition}\newtheorem{rem}[thm]{Remark}
\theoremstyle{definition}\newtheorem{ntn}[thm]{Notation}
\theoremstyle{definition}

\usepackage{amsmath} 
\usepackage{amsthm} 
\usepackage{amsfonts} 
\usepackage{amssymb}
\usepackage{mathrsfs}

\usepackage{tikz}
\usepackage{pgflibraryshapes}

\usepackage{caption}
\usepackage{array}

\usepackage[section]{placeins}

\usetikzlibrary{arrows}
\usetikzlibrary{plothandlers}
\usetikzlibrary{calc}
\usetikzlibrary{snakes}

\newcommand\ssty{\scriptstyle}

\newcommand\ov{\overline}

\newcommand\wh{\widehat}
\newcommand\wt{\widetilde}

\newcommand\mc{\mathcal}
\newcommand\mf{\mathfrak}

\renewcommand\({\left(}
\renewcommand\){\right)}

\newcommand\smeno{\smallsetminus}

\newcommand\0{\underline 0}

\newcommand\nat{\mathbb N} 
\newcommand\ganz{\mathbb Z} 
 
\newcommand\real{\mathbb R}
\newcommand\compl{\mathbb C}

\newcommand\pol{\mathcal P}

\newcommand\conv{\mathrm{Conv}}
\newcommand\cone{\mathrm{Cone}}

\newcommand\gen{\mathrm{Span}}

\newcommand\Min{\mathrm{Min}}
\newcommand\Max{\mathrm{Max}}

\newcommand\rk{\mathrm{rk}}
\newcommand\alt{\mathrm{ht}}

\newcommand\al{\alpha}
\newcommand\be{\beta}
\newcommand\ga{\gamma}

\newcommand\nosim{\ /\hbox{\hskip -7pt $\sim$}\,}

\newcommand\red{\mathrm{Red}}

\newcommand\eo{{\mathrm{e}}}

\newcommand\preq{\preccurlyeq}

\title{Triangulations of root polytopes}

\author{Paola Cellini}

\address{Dipartimento di Ingegneria e Geologia, Universit\`a di Chieti e Pescara, Viale Pindaro 42, 65127 Pescara, Italy}

\email{pcellini@unich.it}

\begin{document}

\maketitle

\begin{abstract}
Let $\Phi$ be an irreducible crystallographic root system and $\pol$ its root polytope, i.e., its convex hull. 
We provide a uniform construction, for all root types, of a triangulation of the facets of $\pol$. We also prove that, on each orbit of facets under the action of the Weyl gruop, the triangulation is unimodular  with respect to a root sublattice that depends on the orbit.
\end{abstract}

\section{Introduction}\label{intro}

\par
Let $\Phi$ be an irreducible crystallographic root system in a Euclidean space $\mathrm E$, $\Phi^+$ a positive system of $\Phi$, and $W$ the Weyl group of $\Phi$. 
We denote by $\pol$ the {\it root polytope} associated with $\Phi$, i.e. the convex hull of all  roots in $\Phi$. 

\par 
In \cite{CM1} we study a natural set of representatives of the faces of $\pol$ modulo the action of $W$, that we call the standard parabolic faces of $\pol$.  
The set of all roots contained in a standard parabolic face is an abelian ideal of  $\Phi^+$ (see Subsection ~\ref{subsec-adnilpotent} for a definition). 
We call {\it face ideals} or {\it facet ideals} the  abelian ideals of $\Phi^+$ corresponding to the standard parabolic faces or facets of $\pol$.

\par
In \cite{CM2}, for $\Phi$ of type $\mathrm A_n$ and $\mathrm C_n$, we have constructed a triangulation of the standard parabolic facets whose simplexes have a natural interpretation in terms of the corresponding facet ideals. 
The construction is formally equal for both root types, though the proofs are distinct and based on the special combinatorics of these two root systems and their maximal abelian ideals.  
Clearly, through the action of $W,$ a triangulation of all the standard parabolic facets can be extended to a triangulation of the boundary of $\pol$. 
Such an extension is not unique and corresponds to a  choice of representatives of the left cosets of $W$ modulo  the stabilizers of the standard parabolic facets.
The triangulations of the boundary of $\pol$ are also studied in  \cite{Mesz1, Mesz2},  for $\mathrm C_n$, and in \cite{Ard} for all classical root types, using the coordinate description of $\Phi$.
In \cite{GGP}, the triangulations of the positive root polytope $\pol^+$, i.e the convex hull of the positive roots and the origin,  are studied for $\Phi$ of type $\mathrm A_n$.

\par
In this paper, we give a uniform construction of a triangulation of the standard parabolic facets,  for all finite irreducible crystallographic root system. 
The construction coincides with the one of \cite{CM2} for the types $\mathrm A_n$ and $\mathrm C_n$.  
We also obtain unimodularity results similar to those obtained for $\mathrm A_n$ and $\mathrm C_n$.  

\par
We need some preliminaries for describing the results in more detail.
If $\be_1,\be_2, \ga_1, \ga_2\in \Phi^+$ are such that $\be_1+\be_2=\ga_1+\ga_2$, we say that $\{\be_1, \be_2\}$ and $\{\ga_1, \ga_2\}$ are {\it crossing pairs}. 
We first prove that if $\{\be_1,\be_2\}, \{\ga_1, \ga_2\}$ are crossing pairs contained in a (common) abelian ideal,  then, for all $i, j$ in $\{1, 2\}$, the differences $\be_i-\ga_j$ are roots, in particular  $\be_i$ and $\ga_j$ are comparable. 
This implies that the set $\{\be_1,\be_2, \ga_1, \ga_2\}$ has a minimum and a maximum, more precisely, one of the two crossing pairs consists of these minimum and maximum, i.e., either $\be_1< \ga_i< \be_2$ for both $i=1$ and $2$, or the analogous relation with $\be$ and $\ga$ interchanged holds.
We define the relations $\lesssim$ and $\sim$ on $\Phi^+$ as follows.
For all $\be_1, \be_2$ in $\Phi^+$, we write $\be_1\lesssim \be_2$ if there exist 
$\ga_1, \ga_2$ such that $\be_1+\be_2=\ga_1+\ga_2$ and $\be_1< \ga_i< \be_2$ for both $i=1$ and $2$. 
Moreover, we write $\be_1\sim \be_2$ if $\be_1\lesssim \be_2$ or $\be_2\lesssim \be_1$.
Finally, we say that a subset $R$ of $\Phi^+$ is {\it reduced} if $\be_1\nosim \be_2$ for all $\be_1, \be_2\in R$.  

\par
The first of main results in this paper is that the maximal reduced subsets in a facet ideal provide a triangulation of the corresponding standard parabolic facet.
For each standard parabolic facet $F$ of $\pol$, let $I_F$ be the corresponding facet ideal:
$$
I_F=F\cap \Phi,
$$
and
$$
\mc T_F=\{\conv(R)\mid R\subseteq I_F,\ R \text { maximal reduced }\},
$$
where $\conv(R)$ is convex hull of $R$. 
Then the following result holds.

\begin{thm}\label{main1}
For each standard parabolic facet  $F$ of $\pol$,  $\mc T_F$ is a triangulation of $F$.
\end{thm}

\nl
Clearly, the set of vertexes of the above triangulation is the set of all roots contained in~$F$. 
\smallskip

\par
Thorem \ref{main1} says, in particular, that the maximal reduced subsets in $I_F$ are linear bases of $\mathrm E$. 
Let $\Pi$ and $\theta$ be the simple system and the highest root of $\Phi^+$. 
Then, $\{-\theta\}\cup \Pi$ is the set of vertexes of the affine Dynkin diagram of $\Phi$.
For each $\al\in \Pi$,  let $\Phi_{\al}$ and $\wt\Phi_{\al}$ be the root subsystems of $\Phi$ generated by $\Pi\smeno\{\al\}$ and  $\{-\theta\}\cup (\Pi\smeno\{\al\})$, respectively, and $\Phi_\al^+$ and $\wt\Phi_{\al}^+$ their positive systems contained in $\Phi^+$.
Clearly, $\wt\Phi_{\al}$ has the same rank as $\Phi$.
We call the $\wt\Phi_\al$, for all $\al\in \Pi$, the {\it standard equal rank} subsystems of $\Phi$.
It is known that the standard parabolic facets of $\pol$ correspond to the {\it irreducible} standard equal rank root  subsystems of $\Phi$. 
In fact, for each $\al\in \Pi$ such that  that $\wt\Phi_{\al}$ is irreducible,  let
$$
I_\al=\wt\Phi_{\al}^+\smeno \Phi_\al.
$$ 
Then $I_\al$ is a facet ideal of $\Phi^+$, and each facet ideal of $\Phi^+$ is obtained in this way (see \cite{CM1}).
We prove the following result. 

\begin{thm}\label{main2}
Let $\al\in \Pi$ be such that $\wt\Phi_{\al}$ is irreducible.
Then, each  maximal reduced subset contained in the facet ideal $I_\al$ is a $\ganz$-basis of the root lattice of $\wt\Phi_{\al}$.  
In particular, all the simplexes of the triangulation  $\mc T_F$  have the same volume.
\end{thm}

\par
Part of the proofs require a case by case analysis. 
The cases to be considered can be restricted to a special, proper subset of facet ideals. 
Indeed, the results of \cite{CM1} imply that the facet ideal  $I_\al$  ($\al\in \Pi$,  $\wt\Phi_{\al}$ irreducible), is an {\it abelian nilradical} (see Subsection \ref{subsec-abnil}) in the root subsystem $\wt\Phi_\al^+$,  hence we may reduce to the case of abelian nilradicals.

\par
The case by case analysis is contained in the proof of Proposition \ref{triang-ord-exist}. 
This proof also provides an algorithm for the explicit computation of the triangulations for each root type, which will be done in a next paper. 

\section{Preliminaries}\label{preli}

In this section we fix our main notation and recall some preliminary results. 
For the basic preliminary notions, we refer to \cite{Bou} and \cite{HuCox} for root systems, and to \cite{Bou2} and \cite{HuLie} for Lie algebras. 

\subsection{Basic notation.}\label{basic}
{\sl General.} We sometimes use the symbol $:=$ for emphasizing  that equality holds by definition or that we are defining the left term of equality.
We denote by $(\ \,,\ )$ the scalar product of $\mathrm E$ and by $|\ \ |$ the corresponding norm. We identify $\mathrm E$ with its dual space, through $(\ \,,\ )$. The null vector of $\mathrm E$ is denoted by $\0$. 
For any $S\subseteq \mathrm E$, $\gen(S)$ is the vector subspace generated by $S$ over $\real$ (the field of real numbers), and $\rk(S):=\dim\gen(S)$.
\par
{\sl Root systems.}
The simple system of $\Phi$ corresponding to the positive system $\Phi^+$ is denoted by $\Pi$,  while $\Omega^\vee$ is the set of fundamental co-weights of $\Phi$, i.e., the  dual basis of $\Pi$ in $\mathrm E$.  
For each $\al\in \Pi$, $\check\omega_\al$ is the fundamental co-weight  defined by the conditions $(\al, \check\omega_\al)=1$ and $(\al', \check \omega_\al)=0$ for all $\al'\in \Pi\smeno \{\al\}$. 
For all $\al\in \Pi$ and $\be \in \Phi$, $c_\al(\be)$ is the coefficient of $\al$ in $\be$, i.e., 
$$c_\al(\be)=(\be, \check\omega_\al).$$
The highest root in $\Phi^+$ is denoted by $\theta$ and its coefficients with respect  to $\Pi$ by $m_\al$, thus
$$\theta=\sum\limits_{\al\in \Pi} m_\al\al.$$ 
We will call $m_\al$ the {\it multiplicity of $\al$ in $\Phi^+$}.
\par

For all $\be\in \Phi$, $\be^\vee$ is the corresponding coroot, i.e., $\be^\vee=\frac{2\be}{(\be, \be)}$.
\par

For each root subsystem $\Psi$ of $\Phi$ we set $\Psi^+=\Psi\cap \Phi^+$. 
It is well known that $\Psi^+$ is a positive system for $\Psi$: we call it the {\it standard positive system} of $ \Psi$.
Moreover, we denote by $L(\Psi)$ and $L^+(\Psi)$ the root lattice and positive root lattice of $\Psi$, i.e. the $\ganz$-span of $\Psi$ and the $\nat$-span of $\Psi^+$, respectively, where $\ganz$ and $\nat$ are the sets of integers and non-negative integers. 

\par
For any $S\subseteq \Phi$, we denote by $\Phi(S)$ the root subsystem of $\Phi$ generated by $S$, i.e., the minimal root system containing $S$,
and we write  $\Phi^+(S)$ for $\Phi(S)^+$. 

\par
{\sl Posets.} As usual, $\leq$ denotes  both the order of $\real$ and the partial order of $\mathrm E$ associated to $\Phi^+$:  for all $x,y\in \mathrm E$,  $x\leq y$ if and only if $y-x\in L^+(\Phi)$.
We call this last order {\it the  standard partial order}.
We will need only the restriction of the standard partial order to $\Phi^+$. 
For any $S\subseteq \Phi^+$, we denote by $\Min\,S$ and $\Max\, S$, with capital M, the sets of minimal and maximal elements of $S$, and by $\min S$ and $\max S$ its possible minimum and maximum, with respect to $\leq$. 
The analogous objects with respect to any other order relation $\preq$, will be distinguished by the subscript~$_\preq$.
The elements in  $\Min\,S\cup \Max\, S$ are called {\it the extremal} elements of $P$.
We say that $S$  is {\it saturated} if it is saturated with respect to the standard partial order, i.e., for all $\be_1, \be_2\in S$ such that $\be_1\leq \be_2$, all the interval $[\be_1, \be_2]:=\{\ga\in \Phi \mid \be_1\leq \ga\leq \be_2\}$ is contained in $S$.
Any subset $S'$ of $S$ is called {\it an initial section of $S$} if for all $\be\in S'$ and $\ga\in S$, if $\ga\preq \be$, then $\ga \in S'$. 
The final sections are defined similarly.
\par
For any order relation $\preq$ on $\Phi^+$ and for all $\be\in \Phi^+$, we denote $(\be^\preq)$ the {\it $\preq$-upper cone of} $\be$, i.e., 
$$(\be^\preq)=\{\ga\in \Phi^+\mid \be\preq \ga\}.$$  
Clearly, this is a dual order ideal, or filter,  in the poset $(\Phi^+,\preq)$.

\subsection{Basic lemmas on roots}\label{subsec-basic}

We say that two roots are summable if their sum is a root. 
It is a basic fact that two roots with negative scalar product are summable and that the converse does not hold, in general. 
\par
Let $\mf g$ be  a complex simple Lie algebra with root system $\Phi$ with respect to the Cartan subalgebra $\mf h$ (see e.g. \cite[\S 18]{HuLie}). 
Thus, $\mf g=\(\bigoplus_{\al\in\Phi}\mf g_\al\)\oplus \mf h$, where $\mf g_\al$ is the root space of $\al$,  for all $\al\in \Phi$, and  $\(\gen_\compl(\Phi)\)=\mf h^*$, the dual space of $\mf h$. 
It is well known that if $\al$ and $\be$ are summable roots, then $[\mf g_\al,\mf g_\be]=\mf g_{\al+\be}$, while if  $\al$ and $\be$ are not summable and $\al\neq -\be$, then  $[\mf g_\al,\mf g_\be]=0$. 

\begin{pro}\label{furbo}
Let $\Phi$ be any crystallographic root system and   let $\be_1, \be_2, \be_3\in  \Phi$ be such that $\be_1+\be_2+\be_3\in \Phi$  and $\be_i\neq -\be_j$ for all $i, j\in\{1, 2, 3\}$.
Then at least two of the three sums $\be_i+\be_j$, with $i, j\in\{1, 2, 3\}$ and $i\neq j$, belong to $\Phi$.
\end{pro}

\begin{proof}
Since $(\be_1+\be_2+\be_3, \be_1+\be_2+\be_3)>0$, at least one of the scalar products $(\be_1+\be_2+\be_3, \be_i)$ with $1\leq i\leq 3$ is strictly positive, whence  $\be_1+\be_2+\be_3- \be_i$ is a root. 
Assume for example $\be_1+\be_2\in \Phi$. 
Then, $[[\mf g_{\be_1}, \mf g_{\be_2}],\mf g_{\be_3}]\neq 0$, hence, by the Jacobi identity, at least one of $[[\mf g_{\be_1}, \mf g_{\be_3}],\mf g_{\be_2}]$ 
and $[\mf g_{\be_1}, [\mf g_{\be_2},\mf g_{\be_3}]]$ is not $0$. It follows that at least one of $\be_1+\be_3$ and $\be_2+\be_3$ is a root.
\end{proof}
\smallskip

In the following Lemma we classify the Cartan integers of pairs of summable roots. 
The proof is an exercise and is omitted. 
The results are well known and will be used also without explicit reference to the lemma. 
\begin{lem}\label{esercizio}
Assume $\be, \ga, \be+\ga\in \Phi$. 
\begin{enumerate}
\item
If $|\be|=| \ga|=|\be+\ga|$, then $(\be, \ga^\vee)=-1$.
\item
If $|\be|=| \ga|\neq |\be+\ga|$, then either 
$\frac{|\be+\ga|^2}{ |\be|^2}= 2$ and  $(\be, \ga^\vee)=0$, or $\frac{|\be+\ga|^2}{ |\be|^2}=3$ and  $(\be, \ga^\vee)=1$. In any case, $|\be|=| \ga|< |\be+\ga|$.
\item
If $|\be|<| \ga|$, then $|\be+\ga|=|\be|$, $(\ga, \be^\vee)=-\frac{|\ga|^2}{ |\be|^2}\in \{-2, -3\}$, and  $(\be, \ga^\vee)=-1$.  
\end{enumerate}
In particular, $(\be, \ga)\geq 0$ if and only if $|\be|=| \ga|< |\be+\ga|$.
\end{lem}

\subsection{Ad-nilpotent and abelian ideals.}\label{subsec-adnilpotent}

\par
Let $\mf g$ be as in Subsection ~\ref{subsec-basic}, 
$\mf b$ be the standard Borel subalgebra of $\mf g$ associated to $\Phi^+$,  and $\mf n$ its nilpotent radical, i.e., $\mf b=\(\bigoplus_{\al\in\Phi^+}\mf g_\al\)\oplus \mf h$ and $\mf n=\bigoplus_{\al\in\Phi^+}\mf g_\al$. 

\par
An {\it ad-nilpotent} ideal of $\mf b$ is a (nilpotent) ideal of $\mf b$ contained in $\mf n$. 
It is clear that such an ideal  is a sum of root spaces. 
For any $I\subseteq \Phi^+$, the sum of root spaces $\bigoplus_{\al\in I} \mf g_\al$ is an ad-nilpotent ideal of $\mf b$ if and only if, for all $\al, \beta\in \Phi^+$, if $\al\in I$ and $\al\leq \beta$, then $\beta\in I$. 
A subset $I$ of $\Phi^+$ with this property is called an ad-nilpotent ideal of $\Phi^+$.
Clearly, this is filter in $(\Phi, \leq)$, i.e. a dual order ideal. 
It is easy to see that an abelian ideal of $\mf b$ must be ad-nilpotent. 
For any  $I\subseteq \Phi^+$, the subspace $\bigoplus_{\al\in I} \mf g_\al$ is an abelian ideal of $\mf b$ if and only if $I$ is an ad-nilpotent ideal of $\Phi^+$ with the further property that, for all $\al, \beta\in I$, $\al+\beta\not\in \Phi$. 
Such an $I$ is called an abelian ideal of $\Phi^+$.
The  abelian ideals of $\Phi^+$ are studied in several papers, both for their implications in representation theory and for their proper algebraic-combinatorial interest. 
The main representation theoretic motivations  can be found in in \cite{Ko1, Ko2} (see also \cite{CMP}); the basic algebraic-combinatorial results can be found \cite{CPadnilp1}, \cite{CPab}, \cite{Pan}, \cite{Su}.

\subsection{Abelian nilradicals.}\label{subsec-abnil}

\par
An ad-nilpotent ideal of $\Phi^+$ is called {\it principal} if it has a minimum, i.e. if the corresponding $\mf b$-ideal is principal. 
For all $\be\in \Phi^+$, the upper $\leq$-cone of $\be$, $(\be^\leq)=\{\ga\in \Phi^+\mid \be\leq \ga\}$ is also called the principal ad-nilpotent ideal generated by $\be$.  
It is clear that if $\be\in \Phi^+$ is such that $c_\al(\be)>\frac{m_\al}{2}$ for some $\al\in \Pi$, then $(\be^\leq)$ is abelian. 
In particular, this happens if $\be$ is a simple root of  multiplicity $1$ in $\Phi^+$. 
Indeed, the following well known result holds. 
The proof is brief, so we include it.
\par

\begin{pro}\label{simpleprincipal}
Let $S\subseteq\Pi$ and $I=\Phi^+\smeno \Phi(S)$. 
Then $I$ is an ad-nilpotent ideal. 
Moreover, $I$ is abelian if and only if either $S=\Pi$, or $S=\Pi\smeno\{\al\} $ for  a simple root $\al$ such that $m_\al=1$. 
In this case, $I$ is equal to $(\alpha^\leq)$ and is a maximal abelian ideal.
\end{pro}

\begin{proof}
It is clear that in any case $I$ is an ad-nilpotent ideal. 
For $S=\Pi$ we obtain the empty root ideal. 
Let  $\Pi\smeno S=\{\al\}$ with $\al\in \Pi$ and $m_\al=1$. 
Then by definition we have $I=(\al^\leq)$, and it is clear that  this is abelian.
It remains to prove that it is maximal. 
If $S= \emptyset$, then $I=\Phi^+$ and the claim is obvious. 
If $S\neq \emptyset$,  any nilpotent ideal  $J$ that strictly contains $I$ also contains the highest root of $\Phi(S)$, and it is clear that this last root is summable to~$\al$.
Hence, $I$ is in any case a maximal abelian ideal.
\par
Now, for all $\be\in \Phi$, let $\alt_{\Pi\smeno S}(\be)=\sum\limits_{\al\in\Pi\smeno S}c_\al(\be)$.
It is clear that the condition $\max\{\alt_{\Pi\smeno S}(\be)\mid\be\in \Phi\}= 1$ is equivalent to $\Pi\smeno S=\{\al\}$ and $m_\al=1$. 
In order to conclude the proof, we assume $\max\{\alt_{\Pi\smeno S}(\be)\mid\be\in \Phi\}> 1$ and  prove that in this case $I$ is not abelian.
Let $\be^*\in \hbox{$\Min\{\be\in \Phi\smeno\Phi(S)\mid \alt_{\Pi\smeno S}(\be)>1\}$}$. 
Then $(\be^*, \al')\leq 0$ for all $\al'\in S$ and, since $(\be^*, \be^*)> 0$, we must have $(\be^*, \al)>0$ for some  $\al\in  \Pi\smeno S$.  
For such an $\al$, $\be^*-\al$ is a root and belongs to $I$, since $\alt_{\Pi\smeno S}(\be^*-\al)>0$. 
But $\be^*-\al$ is summable to $\al$, that also belongs to $I$,  hence $I$ is not abelian. 
\end{proof}

\par
For each $S\subseteq \Pi$, $\bigoplus\limits_{\al\in \Phi^+\smeno \Phi(S)}\mf g_\al$ is the nilradical (the largest nilpotent ideal) of the standard parabolic subalgebra  associated to $S$ (see \cite[Ch. VIII, \S 3.4]{Bou2}). 
Hence, we will call the maximal abelian ideals $(\al^\leq)$ with $m_\al=1$, together  with the empty root ideal, {\it the abelian nilradicals}. 

\subsection{The faces of the root polytope.}\label{subsec-faces}
We recall some ideas and results from \cite{CM1}. 
For all $\al\in \Pi$ and  all $S\subseteq \Pi$,  let
$$
H_{\al, m_\al}=\{x\in E\mid(x, \check\omega_\al)=m_\al\}, \qquad 
\mathrm F_\al=H_{\al, m_\al}\cap \pol,
\qquad
\mathrm F_S=\bigcap_{\al\in S} \mathrm F_\al.
$$
It is clear that the hyperplanes  all $H_{\al, m_\al}$ are supporting hyperplanes of $\pol$, hence the $\mathrm F_\al$ and  $\mathrm F_S$ are faces of $\pol$.  
We call them the {\it standard parabolic faces}. 
In fact, the set of all standard parabolic faces is a set of representatives of the orbits of the action of the Weyl group $W$ on the set of all faces of $\pol$ \cite{CM1}. 

\par
For each standard parabolic face $F$, let 
$$
I_F=F \cap \Phi.
$$
By definition,  for each $S\subseteq \Pi$, $I_{\mathrm F_S}$ is the set of all roots $\be$ such that $c_\al(\be)=m_\al$, for all $\al\in S$. 
It is easy to see that $\pol$ is the convex hull of the long roots (see  e.g. \cite{CM3}), hence the long roots in $I_{\mathrm F_S}$ are the vertexes of the face $\mathrm F_S$. 
\par

For each $S\subseteq\Pi$, let 
$$ S^\eo=\{\theta\}\cup  -S.$$ 
Moreover, let  ${S^\eo_\theta}$ be the subset of $S^\eo$ defined by the condition that $\Phi({S^\eo_\theta})$ is the irreducible component of $\Phi(S^\eo)$ containing $\theta$. 
Finally, let $S_\theta={S^\eo_\theta}\smeno \{\theta\}$.
\par

It is clear that $\Pi^\eo$ is the set of vertexes of the extended Dynkin graph of $\Phi$ with respect to the simple system $-\Pi$.  
We will call this extended Dynkin graph {\it the opposite } extended Dynkin graph (of $\Phi$).
Thus, by definition, the subgraph induced by ${S^\eo_\theta}$  in the opposite extended Dynkin graph is the connected component of $\theta$.
\par
It is immediate from the theory of affine root system, and very easy to see directly, that, for each {\it proper} subset $S$ of $\Pi$, $S^\eo$  is a simple system for the root subsystem $\Phi(S^\eo)$ that it generates.
\par

The following proposition contains the preliminary results on the standard parabolic faces that we need. 
We note that the proposition also precises that the face $F_S$ does not determine $S$. 
Indeed, by item (1) for all $S, S'\subseteq \Pi$, $F_S=F_{S'}$ if and only if  $\Phi^+((\Pi\smeno S)^\eo_\theta) = \Phi^+((\Pi\smeno S')^\eo_\theta)$, 
i.e., $F_S$ is uniquely determined by the irreducible component $\Phi^+((\Pi\smeno S)^\eo_\theta)$.
In particular, the standard parabolic faces, and therefore the $W$-orbits of faces, are in bijection with the proper connected subgraphs of the opposite extended Dynkin graph that contains the vertex $\theta$ \cite{V}.

\vbox{
\begin{pro}\cite{CM1}\label{teo-facce}
Let $S\subseteq \Pi$, $S\neq\emptyset$.
\begin{enumerate}
\item 
$I_{\mathrm F_S}=\Phi^+((\Pi\smeno S)^\eo)\smeno \Phi(\Pi\smeno S)=\Phi^+((\Pi\smeno S)^\eo_\theta)\smeno \Phi((\Pi\smeno S)_\theta)$. 
\item
Let $\mu_S$ be the highest root of $\Phi((\Pi\smeno S)^\eo_\theta)$, with respect to the simple system  $(\Pi\smeno S)^\eo_\theta$. Then,
$I_{\mathrm F_S}$  is the principal abelian ideal of $\Phi^+$ generated by $\mu_S$.
\item
$\dim(\mathrm F_S)=|(\Pi\smeno S)_\theta|$.
\end{enumerate}
\end{pro}
}
By definition of $I_{\mathrm F_S}$, statement (2)  says that $\mu_S$ is the unique minimal root  such that $c_\al(\mu_S)=m_\al$ for all $\al\in S$.
Both (1) and (2) implies that  we have $c_\al(\mu_S)<m_\al$ if and only if $\al\in (\Pi\smeno S)_\theta$. 
Hence, for all $\be\in \Phi^+$, the condition $c_\al(\be)=m_\al$ for all $\al\in S$ implies $c_\al(\be)=m_\al$ also for all  $\al\in \Pi\smeno (\Pi \smeno S)_\theta$, which in general is greater than $S$. 

\begin{rem}\label{rem-teo-facce-2}
Let $\wt{\Phi^+}((\Pi \smeno S)^\eo)$ be the positive system of $\Phi((\Pi \smeno S)^\eo)$ relative to  the simple system $(\Pi \smeno S)^\eo$.
It is clear that for $S\neq \emptyset$ this positive system is different from $\Phi^+((\Pi \smeno S)^\eo)$, which  is the intersection $\Phi(\Pi \smeno S)\cap \Phi^+$, by definition.  
However, $\Phi^+((\Pi \smeno S)^\eo)\smeno \Phi(\Pi \smeno S)=$ \hbox{$\wt{\Phi^+}((\Pi \smeno S)^\eo)\smeno \Phi(\Pi \smeno S)$}. 
The same holds with $(\Pi \smeno S)_\theta$ in place of~$\Pi \smeno S$. 
Therefore, in Proposition~\ref{teo-facce} (1) we may replace  $\Phi^+$ with  $\wt {\Phi^+}$.
\end{rem}

By the above remark, Proposition~\ref{teo-facce} (1) is equivalent to the following corollary.

\begin{cor}\label{cor-teo-facce}
The set $I_{\mathrm F_S}$ is the principal ideal generated by $\theta$  in the positive system~$\wt{\Phi^+}((\Pi \smeno S)^\eo_\theta)$ of the irreducible root system $\Phi((\Pi \smeno S)^\eo_\theta)$. 
\end{cor}

\subsection{The order involution of face ideals}\label{subsec-face-inv}

For all $w\in W$, let $$N(w)=\{\be\in \Phi^+\mid w(\be)\leq \0\}.$$
For all $S\subseteq \Pi$, let $w_{0, S}$ be the longest element in the standard parabolic subgroup of $W$ generated by $\{s_\al\mid \al\in S\}$.
It is well known that $w_{0, S}$ is an involution and is determined by the condition $N(w_{0, S})=\Phi^+(S)$.

\begin{pro}\label{face-inv}
Let $S\subseteq \Pi$ and $w^*_{S}=w_{0,(\Pi\smeno S)}$. 
Then, the restriction of $w^*_{S}$ to $I_{\mathrm F_S}$ is an anti-isomorphism of the poset $(I_{\mathrm F_S}, \leq)$. 
In particular, $w^*_{S}$ exchange $\theta$ and $\mu_S$.
\end{pro}

\begin{proof}
By definition, $I_{\mathrm F_S}=(\theta+L(\Phi(\Pi\smeno S)))\cap \Phi$ and,
obviously, for all $\al\in \Pi\smeno S$, $s_\al(\theta)\in \theta+L(\Phi(\Pi\smeno S))$. Hence it is clear that  $w^*_{S}$ acts on $I_{\mathrm F_S}$. 
\par
It remains to prove that $w^*_{S}$ reverses the standard partial order on $I_{\mathrm F_S}$.
Let $\be, \be'\in I_{\mathrm F_S}$ and $\be< \be'$. 
Then  $\be'-\be\in L^+(\Phi(\Pi\smeno S))$, and  
since $w^*_{S}(\al)< \0$ for all $\al\in (\Pi\smeno S)$, $w^*_{S}(\be')-w^*_{S}(\be)=w^*_{S}(\be-\be')\in -L^+(\Phi(\Pi\smeno S))$, i.e.  $w^*_{S}(\be') < w^*_{S}(\be)$. 
\end{proof}

We note that, by Proposition~\ref{teo-facce}(1), the above proposition holds also with  $w_{0,(\Pi\smeno S)_\theta}$ in place of $w^*_{S}$. 
In particular, the restrictions of $w_{0,(\Pi\smeno S)_\theta}$ and of $w^*_{S}$ on $I_{\mathrm F_S}$ coincide.
\par

\begin{defi}
We call $w^*_S$ the {\it face involution} of $F_S$ and the restriction of $w^*_S$
to $I_{\mathrm F_S}$ the {\it order involution} of  $I_{\mathrm F_S}$. 
\end{defi}

\section{Face ideals and abelian nilradicals}\label{sec-face-ideals}

In this section we prove that the abelian nilradicals of $\Phi^+$ are facet ideals and that all face ideals are abelian nilradicals in some ireducible subsystem of $\Phi$.  
\par

By Proposition~\ref{teo-facce}, the standard parabolic facets of $\pol$ are the faces of type $\mathrm F_{\al}$ with $\al\in \Pi$ such that  $\Phi((\Pi\smeno\{\al\})^\eo)$ is irreducible. 
Equivalently, a face $\mathrm F_{\al}$ ($\al\in \Pi$) is a facet if and only if $\al$ is a leaf of the extended Dynkin diagram. 
In next results we prove that this happens, in particular, if $m_\al=1$. 

\begin{pro}\label{ab-nil-facet}
Each nonempty abelian nilradical of $\Phi^+$ is a facet ideal.
\end{pro}

\begin{proof}
It is well known that if $\al$ is any simple root such that $m_\al=1$,  then the subgraph of the extended Dynkin graph obtained by removing $\al$ is isomorphic to the (ordinary) Dynkin graph of $\Phi$ \cite{IM}. 
In particular, $\Phi((\Pi\smeno \{\al\})^\eo)$ is irreducible. 
\end{proof}

We note that the fact that the simple roots $\al$ with $m_\al=1$ are leafs of the extended Dynkin diagram is also a consequence of Proposition ~\ref{face-inv}.
Indeed,  if $m_\al=1$, then $\al$ is the minimum of $I_{\mathrm F_\al}$,
hence, the order involution $w_{0,\Pi\smeno\{\al\}}$ maps $\al$ onto $\theta$. 
Since it also maps $\Pi\smeno\{\al\}$ onto  $-(\Pi\smeno\{\al\})$, it maps $\Pi$ onto the nodes of the opposite extended Dynkin graph minus $-\al$. 
\par

It is clear that the converse of Proposition \ref{ab-nil-facet} is not true, however the following result holds. 

\begin{pro}\label{face-ab-nil}
Each face ideal in $\Phi^+$ is an abelian nilradical of some irreducible root subsystem of $\Phi$. 
\end{pro}

\begin{proof}
By Corollary ~\ref{cor-teo-facce}, any face ideal $I_{\mathrm F_S}$ ($S\subseteq \Pi$) is the principal ideal generated by $\theta$ in $\wt{\Phi^+}((\Pi\smeno S)^\eo)$.
\par
It clear that, for each $\be\in \Phi((\Pi\smeno S)^\eo)$,  in the expression of $\be$ as a linear combination of the base $(\Pi\smeno S)^\eo$, the coefficient of $\theta$ is at most $1$. 
In other words, the  multiplicity of $\theta$, as a simple root in the positive system $\wt{\Phi^+}((\Pi\smeno S)^\eo)$, is $1$.
Hence, the principal ideal generated by $\theta$ in $\wt{\Phi^+}((\Pi\smeno S)^\eo)$ is an abelian nilradical.
\end{proof}

\begin{rem}\label{rem-ab-nil}
Let $\al\in \Pi$ be such that $\mathrm F_\al$ is a facet. 
By Proposition~\ref{teo-facce}, $I_{\mathrm F_\al}$ is also equal to $(\mu_{\{\al\}}^\leq)$, where  $\mu_{\{\al\}}$  is the unique root in $\Phi$ such that $c_\al(\mu_{\{\al\}})=m_\al$ and $c_{\al'}(\mu_{\{\al\}})<m_{\al'}$ for all $\al'\in\Pi\smeno\{\al\}$.  
By Proposition ~\ref{face-inv}, the face involution $w_{\{\al\}}^*$ maps $(\Pi\smeno\{\al\})^\eo$ onto $\{\mu_{\{\al\}}\}\cup (\Pi\smeno\{\al\})$, therefore this last set is a simple system for $\Phi((\Pi\smeno\{\al\})^\eo)$. 
The positive system corresponding to it is the standard positive system $\Phi^+((\Pi\smeno\{\al\})^\eo)$ and, clearly, has $\theta$ as its highest root. 
It is also clear that  the multiplicity of $\mu_{\{\al\}}$, as a simple root in $\Phi^+((\Pi\smeno\{\al\})^\eo)$, is $1$. 
Thus, $I_{\mathrm F_\al}$ is the abelian nilradical generated by $\mu_{\{\al\}}$ in the  positive system $\Phi^+((\Pi\smeno\{\al\})^\eo)$.
\end{rem}

It is clear that the definition of ad-nilpotent and abelian ideals makes sense also in the reducible case. Let $\Psi$ be any finite crystallographic root system, $\Psi_1, \dots, \Psi_k$ be its irreducible components, $\Psi_i^+$ a positive system for $\Psi_i$, for $i=1, \dots, k$, and $\Psi^+=\Psi_1^+\cup\dots\cup \Psi_k^+$. 
Then, by definition, $I$ is an abelian ideal of $\Psi^+$ if and only if $I\cap \Psi_i^+$ is an abelian ideal of $\Psi_i^+$ for all $i\in\{1, \dots, k\}$. 
Moreover,  $I$ is an abelian nilradical of $\Psi^+$ if and only if $I\cap \Psi_i^+$ is an abelian nilradical of $\Psi_i^+$ for all $i\in\{1, \dots, k\}$. This means that $I\cap \Psi_i^+$, is either empty  or a principal ideal generated by a simple root with multiplicity~$1$.

\begin{lem}\label{induzione}
Let $I$ be an abelian nilradical of $\Phi^+$, $H$ a vector subspace in $\mathrm E$, and $\Psi=H\cap \Phi$.
Then $I\cap H$ is an abelian nilradical  of $\Psi^+$.
\end{lem}

\begin{proof}
Let $I=(\al^\leq) $, with $\al\in \Pi$ and $m_\al=1$.  
Let $\Psi_1\dots, \Psi_k$ be the irreducible components of $\Psi$,  $\Pi_{\Psi_i}$ be the simple system of $\Psi_i^+$ for $i=i, \dots, k$, and let $\Pi_{\Psi}=\Pi_{\Psi_1}\cup \dots\cup\Pi_{\Psi_k}$, the simple system of 
$\Psi^+$.
Let $S_\al=\{\be\in \Pi_\Psi\mid c_\al(\be)=1\}$.
It is clear that if $\be\in S_\al\cap \Pi_{\Psi_i}$, then $\be$  has multiplicity $1$ in $\Psi_i^+$. 
Moreover, since the sum of all roots in a fixed $\Pi_{\Psi_i}$ is a root,  for all $i\in \{1, \dots, k\}$, $ S_\al\cap \Pi_{\Psi_i}$ contains at most one root.
Hence, either $S_\al\cap \Pi_{\Psi_i}=\emptyset$, in which case  $I\cap\Psi_i=\emptyset$, or $S_\al\cap \Pi_{\Psi_i}=\{\be_i\}$ for a certain root $\be_i$ with multiplicity $1$ in $\Psi_i^+$.  
Then, clearly,  $I\cap\Psi_i=(\be_i^\leq)$, hence $I\cap\Psi_i$ is an abelian nilradical of $\Psi_i^+$.
\end{proof}

\section{Crossing pairs}\label{crossingpairs}

In this section we analyze the properties of {\it crossing pairs} contained in abelian ideals. 
In the simply laced case, many of the results that we are proving could be proved in a very simpler way.

\begin{defi}\label{defcrossing}
Let $\be_i, \ga_i\in \Phi, \ i=1, 2$, with $\be_i\neq \ga_j$ for all $i, j\in \{1, 2\}$. 
We say that 
$\{\be_1, \be_2\}$ and  $\{\ga_1, \ga_2\}$  are {\it crossing pairs}
if $\be_1+\be_2=\ga_1+\ga_2$. In this case we call the equality $\be_1+\be_2=\ga_1+\ga_2$ {\it a crossing relation}. We do not assume that $\be_1\neq\be_2$ and 
$\ga_1\neq\ga_2$, hence (at most) one of the pairs $\{\be_1, \be_2\}$ and $\{\ga_1, \ga_2\}$ may be a multiset of a single root with multiplicity 2. 
\end{defi}

\begin{lem}\label{corte-abeliani}
Let $I$ be an abelian ideal in $\Phi^+$. 
\begin{enumerate} 
\item
If $\be\in I\cap \Phi_s$, $x\in \Phi$, and $\be+x\in \Phi$, then $x\in \Phi_s$.
\item
If $\be, \ga\in I$ and $\be-\ga\in\Phi$, then $(\be,\ga)> 0$.
\end{enumerate}
\end{lem}

\begin{proof}
(1) By contradiction,  if $x\in\Phi_\ell^+$, then by Lemma ~\ref{esercizio}(3) $(\be,x)<0$, hence $s_\be(x)=x +\frac{|x|^2}{|\be|^2}\be \geq x +2\be$. 
It follows $x +2\be\in \Phi$, hence $x +2\be\in I$. 
This is impossible since $x+2\be=\be+(x+\be)$ and $I$ is abelian.
\par
(2) By Lemma ~\ref{esercizio}, if $(\be, \ga)\leq 0$, then $\be, \ga\in\Phi_s$ and $\be-\ga\in\Phi_\ell$: by part (1) this is impossible.
\end{proof}
\smallskip

\vbox{
\begin{pro}\label{crossing}
Let $I$ be an abelian ideal in $\Phi^+$ and $\{\be_1, \be_2\}$,  $\{\ga_1, \ga_2\}$ be crossing pairs contained in $I$ and such that $\be_1\ne \be_2$.
Then:
\begin{enumerate}
\item
for all $i,j\in \{1, 2\}$
$(\be_i, \ga_j)> 0$, in particular 
$\be_i-\ga_j$ is a root; 
\item 
either $\{\be_1, \be_2\}$, or $\{\ga_1, \ga_2\}$ is the pair of the minimum and maximum of \hbox{$\{\be_i, \ga_i\mid i=1,2\}$}.
\item $(\be_1, \be_2)=0$ unless both of $\be_1$, $\be_2$ are short and $\ga_1$, $\ga_2$ have different lengths;
\end{enumerate}
\end{pro}
}
\begin{proof}
(1) 
For $i\in \{1, 2\}$, $\be_1+\be_2-\ga_i\in \Phi$, and since $I$ is abelian, $\be_1+\be_2\not\in \Phi$.  
By Lemma ~\ref{furbo}, we obtain $\be_j-\ga_i\in \Phi$ for $j\in\{1, 2\}$.
By Lemma ~\ref{corte-abeliani}, it follows $(\be_j, \ga_i)> 0$ for $i,j\in \{1, 2\}$.
\par

(2) We set $x=\ga_1-\be_1=\be_2-\ga_2$ and $y=\ga_2-\be_1=\be_2-\ga_1$. 
By (1),  $x$ and $y$ are roots. 
If $x$ and $y$ are  both positive or both negative, we directly obtain that  $\{\be_1, \be_2\}$ is the set of  the minimum and maximum  of $\{\be_i,\ga_1\mid i=1,2\}$.
Similarly, if one of $x, y$ is positive and the other is negative,  $\{\ga_1, \ga_2\}$ is the set of  the minimum and maximum $\{\be_i,\ga_1\mid i=1,2\}$. 
In the picture below the proof we illustrate the Hasse diagram of the quadruple $\{\be_1, \be_2, \ga_1, \ga_2\}$ in the cases $x, y> \0$ and $x> \0, y< \0$. 
\par
(3)
We keep the notation of (2).
We first assume that at least one of $\be_1$, $\be_2$, is long.
Let  $\be_1$ be long. 
Then, by (1) and Lemma ~\ref{esercizio}, we have $(\be_1^\vee, \ga_2)=-(\be_1^\vee, x)=1$, hence $(\be_1^\vee, \be_2)=(\be_1^\vee, \ga_2+ x)=0$.
The case $\be_2$ long is similar, so we assume that both   $\be_1$ and  $\be_2$  are short and $(\be_1^\vee, \be_2)\neq 0$. 
Then,  $1=(\be_1^\vee, \be_2)=(\be_1^\vee, \be_1)+(\be_1^\vee, x)+(\be_1^\vee, y)=2+(\be_1^\vee, x)+(\be_1^\vee, y)$. By Lemma ~\ref{corte-abeliani}, $x$ and $y$ are short, hence one of $(\be_1^\vee, x)$ and $(\be_1^\vee, y)$ is $0$ and the other is $-1$. 
By Lemma ~\ref{esercizio}, this implies that one of $\ga_1$ and $\ga_2$ is long and the other is short. 
\end{proof}

$$
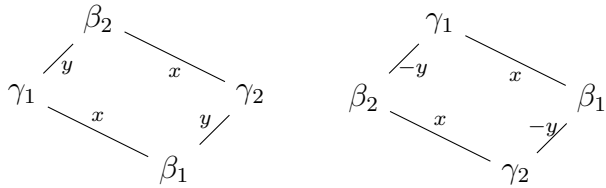
\begin{tikzpicture}[scale=1]
\draw
(0, 0) node(be2) {$\be_2$} 
(2,-1) node(ga2) {$\ga_2$}
(-1,-1) node(ga1){$\ga_1$}
(1,-2) node(be1) {$\be_1$}
(be2)
--node[pos=0.5, below]{$\ssty x$}
(ga2)
--node[pos=0.8, above]{$\ssty y$} 
(be1)
--node[pos=0.5, above]{$\ssty x$}
(ga1)
--node[pos=0.8, below]{$\ssty y$}
(be2);
\end{tikzpicture}
\qquad
\begin{tikzpicture}[scale=1]
\draw
(0, 0) node(ga1) {$\ga_1$} 
(2,-1) node(be1) {$\be_1$}
(-1,-1) node(be2){$\be_2$}
(1,-2) node(ga2) {$\ga_2$}
(be2)
--node[pos=0.5, above]{$\ssty x$}
(ga2)
--node[pos=0.2, above]{$\ssty -y$} 
(be1)
--node[pos=0.5, below]{$\ssty x$}
(ga1)
--node[pos=0.2, below]{$\ssty -y$}
(be2);
\end{tikzpicture}
$$

By the above result, we may define the relations below.

\smallskip
\begin{ntn}
We write $\be_1< \{\ga_1, \ga_2\}< \be_2$ for  $\be_1< \ga_i<\be_2$ for both $i\in \{1, 2\}$.
\end{ntn} 
\smallskip

\begin{defi}\label{lesssim}
We define the relations $\lesssim$ and $\sim$ on $\Phi^+$ as follows: 
\par
$\be_1\lesssim\be_2$
if and only if there exists $\ga_1, \ga_2\in \Phi^+$ such that 
 $\{\be_1, \be_2\}$ and $\{\ga_1, \ga_2\}$ are crossing pairs with $\be_1< \{\ga_1, \ga_2\}< \be_2$;
\par 
$\be_1\sim\be_2$
if and only if either $\be_1\lesssim \be_2$  or $\be_2\lesssim \be_1$.
\par
If  $\{\be_1, \be_2\}$ and $\{\ga_1, \ga_2\}$ are crossing pairs with $\be_1< \{\ga_1, \ga_2\}< \be_2$, we also say that $\{\ga_1, \ga_2\}$ is a {\it middle pair between} $\be_1$ and $\be_2$ and that $\{\be_1, \be_2\}$ is a {\it raising pair through} $\ga_1$ and $\ga_2$.
\end{defi}
\smallskip

Next corollary precises the order relation among different raising pairs through a common middle pair and  different middle pairs between a common raising pair. 
\begin{cor}\label{catene-crossing}
Let $I$ be an abelian ideal, $\{\be_1, \be_2\}$ and $\{\ga_1, \ga_2\}$ be crossing pairs in $I$ with $\be_1< \{\ga_1, \ga_2\}< \be_2$. 
\begin{enumerate}
\item
If $\{\be'_1, \be'_2\}$ is any other raising pair through $\{\ga_1, \ga_2\}$, then either $\be_1< \be'_1< \be'_2< \be_2$, or $\be'_1< \be_1< \be_2< \be'_2$. Moreover,  $\be_i-\be'_i\in \Phi$ for both $i=1, 2$.
\item
If $\{\ga'_1, \ga'_2\}$ is any other middle pair between $\{\be_1, \be_2\}$, then
$\ga_i-\ga'_j\in \Phi$ for all $i,j\in \{1, 2\}$. 
Moreover, one of the following four cases occur: $\ga'_i< \{\ga_1, \ga_2\}< \ga'_j$, $ \ga_i< \{\ga'_1, \ga'_2\}< \ga_j$ (with $\{i, j\}=\{1, 2\}$). 
In particular, there exists at most one incomparable middle pair between 
$\be_1$ and $\be_2$.
\end{enumerate}
\end{cor}

\begin{proof}
Under the assumption of (1), $\{\be'_1, \be'_2\}$ and $\{\be_1, \be_2\}$ are crossing pairs. Similarly, under the assumption of (2),  $\{\ga'_1, \ga'_2\}$ and $\{\ga_1, \ga_2\}$ are crossing pairs. Hence the claim  follows directly from Proposition ~\ref{crossing}.  
\end{proof}

In next Lemma, we see that the possible lengths of roots and root differences in a crossing pair are very limited.

\begin{lem}\label{lati-crossing}
Let $I$ be an abelian ideal in $\Phi^+$, $\{\be_1, \be_2\}$,  $\{\ga_1, \ga_2\}$ be crossing pairs contained in $I$, $\be_1 < \{\ga_1,\ga_2\} < \be_2$, 
$x=\be_2-\ga_1$, and $y=\be_2-\ga_2$.  
\begin{enumerate}
\item
If $x$ is long, then also $y$, $\be_1$, $\be_2$, $\ga_1$, $\ga_2$ are long. 
\item
If any of $x$, $y$, $\be_1$, $\be_2$, $\ga_1$, $\ga_2$ is short,  then $x$ and $y$ are short and at most one of $\be_1$, $\be_2$, $\ga_1$, $\ga_2$ is long, except when $\ga_1=\ga_2$, in which case $\ga_1$ is short and $\be_1, \be_2$ are long. 
\end{enumerate}
\end{lem}

\begin{proof}
We first prove that if one of $x$, $y$ is short, then the other is short, too.
Assume, for example, $x\in\Phi_s$. If $\be_1\in\Phi_s$, then $y\in\Phi_s$ by Lemma ~\ref{corte-abeliani}, hence let $\be_1\in\Phi_\ell$. In this case, by Lemma ~\ref{esercizio}, $\be_1+x=\ga_2\in\Phi_s$, whence $y\in\Phi_s$  by Lemma ~\ref{corte-abeliani}.
\par
Hence, $x, y$ are either both short, or both long.
In order to prove (1), it remains to check that if $x, y\in \Phi_\ell$, then 
$\be_i, \ga_j\in \Phi_\ell$ for $i=1, 2$. 
This follows directly from Lemma ~\ref{corte-abeliani} for $\be_1$, $\ga_1$, and $\ga_2$. 
For $\be_2$ it follows from Lemma ~\ref{esercizio}, since $\be_2=\ga_1+x$.
\par
It remains to conclude the proof of (2). 
By (1), if any of $x$, $y$, $\be_1$, $\be_2$, $\ga_1$, $\ga_2$ belongs to $\Phi_s$, then $x, y\in \Phi_s$.  
In this case, assume $\be_i\in \Phi_\ell$ for a certain $i\in \{1, 2\}$, and let $\{i'\}=\{1, 2\}\smeno\{i \}$. 
For $j\in \{1,2\}$,  $\be_i-\gamma_j\in \{\pm x, \pm y\}$, hence  by Lemma ~\ref{esercizio},  $\ga_j\in \Phi_s$. 
If also $\be_{i'}\in \Phi_\ell$ then, by Lemma ~\ref{esercizio}, $(\ga_i, x)=(\ga_i, y)=0$ for both $i=1,2$, hence $(\ga_2, \ga_1^\vee)=(\ga_1 +x-y, \ga_1^\vee)=2$, which implies $\ga_1=\ga_2$, since $|\ga_1|=|\ga_2|$.  
Conversely,  if $\ga_1=\ga_2$, then we get $-\be_{i'}=\be_i-2\ga_1=s_{\ga_1}(\be_i)$, where $s_{\ga_1}$ is the reflection with respect to $\ga_1$, hence 
$|\be_{i'}|=|\be_i|$. 
\par
By a similar argument, taking into account that $\be_1\neq \be_2$, we obtain that if one of $\ga_1$, $\ga_2$ is long, all the remaining roots in the crossing pairs are short.
\end{proof}

In next proposition, we prove that, for comparable roots $\be_1$ and $\be_2$ in an abelian ideal $I$, if $\be_1-\be_2$ is  not a root, then $\be_1\sim \be_2$. 
Moreover, we analyze when the reverse implication holds. 
In particular, we see that this happens when $\be_1$ and $\be_2$ are long roots, hence in the simply laced case, i.e., in this case we have $\be_1\sim \be_2$ if and only if $\be_1-\be_2\not\in \Phi_\ell$. 

\begin{pro}\label{minori-crossing}
Let $I$ be an abelian ideal in $\Phi^+$ and $\be_1, \be_2\in I$. 
\begin{enumerate}
\item 
If $\be_1<\be_2$ and $\be_2-\be_1\not\in\Phi$, then $\be_1\lesssim\be_2$.
\item 
If $\be_1\lesssim\be_2$, $\{\be_1, \be_2\}\subseteq\Phi_s$ and there exists a middle pair $\{\ga_1, \ga_2\}$ between $\be_1, \be_2$ such that $\ga_1\in\Phi_s$ and $\ga_2\in\Phi_\ell$, then $\be_2-\be_1\in\Phi_s$.
\item
If $\be_1\lesssim\be_2$, then  $\be_2-\be_1\not\in\Phi$ if and only if either of the following conditions is satisfied:
\begin{itemize}
\item[(a)] at least one of $\be_1$, $\be_2$ is long,
\item[(b)] $\{\be_1, \be_2\}\subseteq\Phi_s$ and there exists  a middle pair $\{\ga_1, \ga_2\}\subseteq\Phi_s$ between $\be_1, \be_2$.
\end{itemize}
\end{enumerate}
\end{pro}

\begin{proof}
(1) Let $\be_1< \be_2$ and $\be_2-\be_1\not\in\Phi$. By definition, $\be_2-\be_1$ is a sum of positive roots. Let 
$$
k=\min\{h\in \nat \mid \exists\, \eta_1,\dots, \eta_h\in \Phi^+ \text{ such that }  \be_2-\be_1=\eta_1+\cdots+\eta_h\},
$$ 
and $\eta_1,\dots, \eta_k\in \Phi^+$ be such that $\be_2=\be_1+\eta_1+\cdots+\eta_k$. By assumption, $k\geq 2$ and no sum $\sum_{j=1}^h\eta_{i_j}$ with $1\leq i_j\leq k$ and $h> 1$  is a root.
Clearly, at least one among $(\be_2, \be_1)$, $(\be_2, \eta_i)$ with $1\leq i \leq k$, must be strictly positive. 
Now $\be_2-\be_1\not\in \Phi$ by assumption, and also $\be_1+\be_2\not \in\Phi$, since $I$ is abelian, hence $(\be_1, \be_2)=0$.  
Therefore $(\be_2, \eta_i)> 0$ for some $\in\{1, \dots, k\}$. 
We may assume $(\be_2, \eta_k)> 0$, so that $\be_2-\eta_k=\be_1+\eta_1+\cdots \eta_{k-1}\in \Phi$.
Let $\ga_i=\be_1+\sum\limits_{1\leq j\leq i}\eta_j$: iterating the above argument, we may assume $\ga_i\in \Phi$ for all $i\in \{0, \dots, k\}$. 
Since $\eta_i+\eta_j\not\in \Phi$, for $1\leq i, j\leq k$,  by Proposition ~\ref{furbo} applied to any sum $\ga_i+\eta_{i+1}+\eta_{i+2}$, we get that both  $\ga_i+\eta_{i+1}$ and $\ga_i+\eta_{i+2}$ belong to $\Phi$, for $0\leq i\leq k-2$. 
It follows easily that, for any rearrangement $\eta'_1,\dots, \eta'_k$ of $\eta_1,\dots, \eta_k$, $\ga'_i=\be_1+\sum\limits_{1\leq j\leq i}\eta'_j$ is a root, for $0\leq i\leq k$,  . 
In particular, $\be_1+\eta_1$ and $\be_2-\eta_1$ are roots, both different from $\be_1$ and $\be_2$, hence $\be_1+\be_2=(\be_1+\eta_1)+(\be_2-\eta_1)$ is a crossing relation.
\par
(2)  Let $\be_1, \be_2, \ga_1\in \Phi_s$ and  $\ga_2\in \Phi_\ell$,   $\be_1+\be_2=\ga_1+\ga_2$, $\be_1< \{\ga_1, \ga_2\}< \be_2$, 
$x=\be_2-\ga_1=\ga_2-\be_1$, and $y=\be_2-\ga_2=\ga_1-\be_1$.
Then, by Lemmas ~\ref{corte-abeliani}(2) and ~\ref{esercizio}, we obtain $(\be_2, \be_1^\vee)=(\ga_2+y, \be_1^\vee)=2
+(y, \be_1^\vee)\geq 1$ and hence $(\be_2, \be_1^\vee)=1$ and $\be_2-\be_1\in \Phi_s$.
\par
(3)
Let $\be_1, \be_2, \ga_1, \ga_2\in \Phi^+$ be such that $\be_1+\be_2=\ga_1+\ga_2$, $\be_1< \{\ga_1, \ga_2\}< \be_2$, 
$x=\be_2-\ga_1$, and $y=\be_2-\ga_2$. 
By Lemma ~\ref{lati-crossing}(2), if neither (a) nor (b) hold, then we are in the case of  item (2), hence $\be_1-\be_2\in \Phi$.
It remains to prove the converse. \par
Let $\be_2\in \Phi_\ell$.
Then by Lemma ~\ref{esercizio} $(x, \be_2^\vee)=(y, \be_2^\vee)=1$, hence
$$
2=(\be_2, \be_2^\vee)=(\be_1+x+y, \be_2^\vee)=(\be_1, \be_2^\vee)+2.
\leqno({*)}
$$
It follows $(\be_1, \be_2^\vee)=0$, and $\be_2-\be_1\not\in \Phi$.
The case $\be_1\in \Phi_\ell$ is similar.
\par
If $\be_1,\be_2,\ga_1, \ga_2\in\Phi_s$, then by Lemma ~\ref{corte-abeliani} also $x, y\in \Phi_s$, hence equalities  ($*$) still hold and  we can argue as  above.
\end{proof}
\smallskip

\begin{defi}\label{reduced}
For any  $S\subseteq\Phi^+$,  
we say that $S$ is reduced if, for all $\be, \be'\in S$, $\be\nosim \be'$. 
\par
For all $\be\in \Phi^+$ we set 
$$
\red(\be)=\{\be'\in \Phi^+\mid \be\neq \be' \text{ and } \be\not\sim \be'\}, \qquad  \red(\be)^\leq=\red(\be)\cap (\be^\leq).
$$
\end{defi}
\smallskip

\begin{rem}
By Proposition ~\ref{minori-crossing} and Lemma ~\ref{corte-abeliani}, 
$$\red(\be)^\leq \subseteq \{\eta\in (\be^\leq)\mid \eta-\be\in \Phi^+\}= \{\eta\in (\be^\leq)\mid (\eta,\be)> 0\}.$$
Moreover, in the simply laced case, the inclusion is an equality.
In general, the inclusion is proper. 
As an example, in type $\mathrm C_n$, if we number the simple roots as in \cite{Bou}, and take $\be_1=\al_n+\al_{n-1}, \be_2=\al_n+2\al_{n-1}+\al_{n-2}, \ga_1=\al_n+2\al_{n-1}, \ga_2=\al_n+\al_{n-1}+\al_{n-2}$, we have: $\be_1+\be_2=\ga_1+\ga_2$, hence $\be_1\lesssim \be_2$, but $\be_2-\be_2=\al_{n-1}+\al_{n-2}\in \Phi_s$.
\end{rem}

\section{Triangulation orders}\label{section-triang-order} 

In this section we define some special orderings of abelian ideals, which we call triangulation orders, and prove that all facet ideals have a triangulation order.
Throughout the section, let $I$ be an abelian ideal of $\Phi^+$ such that $\rk(I)=n$.

\begin{defi}\label{bipartite}
Let $J \subseteq I$.
We say that $J$ is bipartite if it has an initial section  $J_{\mathrm i}$, and a final section $J_{\mathrm f}$ such that 
\begin{enumerate}
\item $J=J_{\mathrm i}\cup J_{\mathrm f}$; 
\item for all $\be_1\in J_{\mathrm i}\smeno J_{\mathrm f}$  and $\be_2\in J_{\mathrm f}\smeno J_{\mathrm i}$, we have  $\be_1\lesssim \be_2$;
\item there exists a hyperplane $H$ in $\mathrm E$ such that $J_{\mathrm i}\cap J_{\mathrm f}\subseteq H$ and $H$ strictly separates  $J_{\mathrm i}\smeno J_{\mathrm f}$  from $J_{\mathrm i}\smeno J_{\mathrm f}$.
\end{enumerate}
If the above conditions hold, we say that $\{J_{\mathrm i},J_{\mathrm f}\}$ is a  bipartition of $J$.
If, moreover, $J_{\mathrm i}$ and  $J_{\mathrm f}$ are both proper subsets of $J$, we say that
$\{J_{\mathrm i},J_{\mathrm f}\}$ is a proper bipartition.  
A hyperplane $H$ as in (3) is called a separating hyperplane, for the bipartition  $\{J_{\mathrm i},J_{\mathrm f}\}$ of $J$.
\end{defi}

Note that, by definition, if $J$ has a proper bipartition, then  it has at least two elements. 
If $J$ is also saturated, it has at least three.
From the definition it also follows directly that if  $\{J_{\mathrm i}, J_{\mathrm f}\}$ is a  bipartition of $J$, then $J_{\mathrm i}\smeno J_{\mathrm f}$ and $J_{\mathrm f}\smeno J_{\mathrm i}$ are an initial and a final section of $J$. 
Indeed, for example,  if $\be_1\in J_{\mathrm i}\smeno J_{\mathrm f}$ and $\be\in J\smeno(J_{\mathrm i}\smeno J_{\mathrm f})$ we have $\be\not \leq \be_1$, because $J\smeno(J_{\mathrm i}\smeno J_{\mathrm f})=J_{\mathrm f}$ and $J_{\mathrm f}$ is a final section. 
Moreover, it is clear that  $(J_{\mathrm i}\smeno J_{\mathrm f}) \leq (J_{\mathrm f}\smeno J_{\mathrm i})$.
Finally, we note that if $J$ is saturated, also $J_{\mathrm i}$ and  $J_{\mathrm f}$ are saturated. 

\smallskip

\begin{defi} \label{restricted-sim}
For each subset  $S$ of $\Phi^+$, we define the restricted relations $\lesssim_S$ and $\sim_S$ on $S$ as follows.
For all $\be_1, \be_2\in S$: (1) $\be_1 \lesssim_S \be_2$ if and only if there exists a middle pair $\{\ga_1, \ga_2\}$ between $\be_1$ and $\be_2$ contained in $S$; 
(2) $\be_1 \sim_S \be_2$  and only if either $\be_1 \lesssim_S \be_2$, or $\be_2 \lesssim_S \be_1$.
We say that $S$ is  $\sim$closed if, for all $\be_1, \be_2\in S\cap \Psi^+$,   $\be_1 \lesssim\be_2$ implies   $\be_1 \lesssim_S  \be_2$. 
\end{defi}

It is clear that, for any $S\subseteq \Phi^+$, the relation $\be_1 \lesssim_S \be_2$ implies $\be_1 \lesssim \be_2$, while the converse, in general, does not hold. 
Hence, if $S$ is $\sim$closed, for all  $\be_1, \be_2\in S$ we have $\be_1\sim\be_2$ if and only if $\be_1\sim_S\be_2$.

The first of following lemmas is clear, hence we omit the proof.

\begin{lem}\label{saturated-simclosed}
Let $S\subseteq \Phi^+$. 
If $S$ is saturated, then $S$ is $\sim$closed.
\end{lem}

\begin{lem}\label{simclosed-components}
Let I be an abelian ideal in $\Phi$, $\Psi$ a root subsystem of $\Phi$, and $\Psi_1, \dots, \Psi_k$ be the irreducible components of $\Psi$.
If $I\cap \Psi$ is $\sim$closed, then for all $\be_1, \be_2\in$  \hbox{$I\cap \Psi$}, $\be_1\lesssim \be_2$ if and only if there exists $i\in \{1, \dots, k\}$ such that $\be_1, \be_2\in \Psi_i$ and \hbox{$\be_1\lesssim_{I\cap\Psi_i} \be_2$}.
\end{lem}

\begin{proof}
Let $\be_1, \be_2\in I\cap\Psi$ and $\be_1\lesssim \be_2$. 
If $I\cap \Psi$ is $\sim$closed,  there exists a middle pair $\{\ga_1, \ga_2\}\subseteq  I\cap \Psi$, between $\be_1$ and $\be_2$. 
By Proposition \ref{crossing}, $(\be_i,\ga_j)>0$  for all $i,j\in \{1, 2\}$, hence all of $\be_i$ and $\ga_i$ belong to the same irreducible component of $\Psi$.
\end{proof}

\begin{lem}\label{lemma-simclosed}
Let $I$ be an abelian nilradical of $\Phi^+$, $\Psi$ a parabolic subsystem of $\Phi$, and $\Pi_\Psi$ the simple system of $\Psi^+$.
Assume that:  (1) $\Pi_\Psi\smeno I\subseteq \Pi$; (2) each maximal root in $I\cap \Psi$ is comparable with at most one root in $\Pi_\Psi\cap I$.  
Then $I\cap \Psi$ is saturated,  hence $\sim$closed. 
\end{lem}

\begin{proof}
As seen in the proof of Lemma \ref{induzione}, $I\cap \Psi=\bigcup\limits_{\be\in \Pi_\Psi\cap I}(\be^\preq)$. 
Moreover, different elements in  $\Pi_\Psi\cap I$ belong to different irreducible components of $\Psi$. 

\par
The maximal elements in $I\cap \Psi$ are clearly the highest roots of the irreducible components of $\Psi$ that have nonempty intersection with $I$, hence, condition (2) implies that, for any pair of irreducible components $\Psi_1$ and $\Psi_2$ of $\Psi$, $I\cap\Psi_1$ and $I\cap\Psi_2$ are element-wise  incomparable with respect to the standard partial order $\leq$ of $\Phi$.
Therefore, $I\cap \Psi$ is saturated  if and only if, for each irreducible component $\Psi_1$ of $\Psi$, $I\cap \Psi_1$ is  saturated.
\par
Let $\Psi_1$  be a fixed irreducible component of $\Psi$ such that $I\cap \Psi_1\neq \emptyset$, $\Pi_1=\Pi_\Psi\cap \Psi_1$, and $\be_1, \be_2\in \Psi_1$.
Then, $\be_1-\be_2$ is a $\ganz$-linear combination of roots in $\Pi_1\smeno I$. 
By condition (1), this is contained in $\Pi$, which is $\ganz$-basis of $L(\Phi)$, hence, if $\be_1-\be_2\in L^+(\Phi)$, we obtain that $\be_1-\be_2\in L^+(\Psi_1)$. 
In this case, since $\Psi$, and hence $\Psi_1$, is parabolic, we have also that all roots between $\be_1$ and $\be_2$ belong to $\Psi_1$, hence to  $I\cap\Psi_1$. 
\par
This proves that $I\cap \Psi$ is saturated. 
By Lemma \ref{saturated-simclosed}, it is also $\sim$closed.
\end{proof}

\begin{defi}\label{detachable}
Let $J\subseteq I$, and $\be\in J$.
We say that $\be$ is a detachable element in $J$ if the following conditions hold:
\begin{enumerate}
\item $\be$ is an extremal element of $J$ with respect to the standard partial order; 
\item there exists a hyperplane $H$ such that:
\begin{itemize}
\item[(a)]  $\red(\be)\cap J=J\cap H$ and $H$ strictly separates $\be$ from $J\smeno(\{\be\}\cup \red(\be))$;
\item[(b)]
$I\cap H $ is  $\sim$closed. 
\end{itemize}
We call such a hyperplane $H$ a detaching hyperplane for $\be$ in $J$.
\end{enumerate}
\end{defi}

\begin{rem}\label{rem-detachable}
Let $\be$ be detachable in  $J$, $H$ be a detaching hyperplane, and $J_\be=$ \hbox{$\{\be\}\cup (J\cap H)$}. 
Then,  $J_\be=\{\be\}\cup (J\cap \red(\be))$ and  it is immediate from Definition ~\ref{bipartite} that $\{J_\be, J\smeno\{\be\}\}$ is a bipartition of $J$. 
\end{rem}

\begin{lem}\label{lem-detachable}
Let $\be\in I\cap \Phi_\ell$. 
Then there exist a hyperplane $H$ such that $I\cap H = \red(\be)$, $H$ strictly  separates $\be$ from $I\smeno(\red(\be)\cup \{\be\})$, and $I\cap H$ is $\sim$closed.
In particular, for all $J\subseteq I$ such that $\be=\min J$ or $\be=\max J$,  $\be$ is detachable in  $J$ and $H$ is a detaching hyperplane for $\be$ in $J$. 
\end{lem}

\begin{proof}
Let $\al_I$ be the (unique) simple root such that $I=\{\ga\in \Phi\mid c_{\al_I}(\ga)=m_{\al_I}\}$. 
By Proposition ~\ref{minori-crossing}, for all $\ga\in J\smeno\{\be\}$ we have $\be\not\sim\ga$ if and only if $\be-\ga\in \Phi$. 
Since $\be\in \Phi_\ell$,  this condition is equivalent to $(\be^\vee, \ga)=1$. 
Recall that $\check\omega_{\al_I}$ is the fundamental coweight such that $(\al_I, \check\omega_{\al_I})=1$, and let $\nu=m_{\al_I}\be^\vee-\check\omega_{\al_I}$, $H=\nu^\perp$.
Then, $(\nu, \be)=m_{\al_I}$, and $(\nu, \ga)=0$ for all $\ga$ in $J$ such that $(\be^\vee, \ga)=1$.  
Since $I$ is abelian,  for all other $\ga\in I\smeno\{\be\}$ we have $(\be^\vee, \ga)=0$, hence $(\nu, \ga)=-m_{\al_I}$.
Thus we have proved that $I\cap H = \red(\be)$, and $H$ strictly  separates $\be$ from $I\smeno(\red(\be)\cup \{\be\})$
\par
It remains to prove that $I\cap H$ is $\sim$closed. 
Let $\be_1, \be_2\in \Phi^+\cap H$, $\be_1\sim\be_2$, and $\{\ga_1,\ga_2\}$ be a middle pair between $\be_1$ and $\be_2$.
Then $(\ga_1+\ga_2, \be^\vee)=(\be_1+\be_2, \be^\vee)=2$. 
Since $\be$ is long, this forces $(\ga_1, \be^\vee)=(\ga_2, \be^\vee)=1$, hence
$\{\ga_1,\ga_2\}\subseteq I \cap H$, and $\be_1\sim_{I\cap H}\be_2$, as claimed.
\end{proof}

\begin{defi}\label{triang-order}
Let 
$\preccurlyeq$ be a total order relation on $I$,  
$$
S_{I,\preq}=\{\be\in I\mid \rk(\be^\preq)=n\}.
$$ 
We say that $\preq$ is {\em a triangulation order} if the following conditions hold:
\begin{enumerate}
\item
$I\cap \gen(I\smeno S_{I,\preq})$ is saturated;
\item
for each $\be\in S_{I,\preq}$,  $(\be^\preq)$ is saturated and either  of the following conditions holds:
\begin{itemize}
\item[(a)] $\be$ is detachable in  $(\be^\preq)$,  
\item[(b)] $(\be^\preq)$ has a bipartition $\{J_{\mathrm i}, J_{\mathrm f}\}$ such that,  for both $J=J_{\mathrm i}$ and $J_{\mathrm f}$, $\be$ is detachable in $J$, and there exist a detaching hyperplane $H_J$, for $\be$ in $J$, such that $(\be^\preq) \cap H_J\subseteq \red(\be)$.
\end{itemize}
\end{enumerate}
\end{defi}

\begin{rem}\label{ini-sec}
\begin{enumerate}
\item
It is clear that, for any total ordering $\preq$  on $I$, the subset $S_{I,\preq}$
is an initial section of the ordered set $(I, \preq)$.
Moreover, $\rk(I\smeno S_{I,\preq})< n$.
\item 
It may happen that $I\smeno S_{I,\preq}$ be properly contained in $I\cap \gen(I\smeno S_{I,\preq})$. 
In fact, this happens for a triangulation order that we will construct  for type $\mathrm E_{7}$.
\item
The above definition does not contain any condition on the restriction of $\preq$ to $I\smeno S_{I,\preq}$. 
Hence, if $\preq$ is a triangulation order, any other total order $\preq'$ such that  $S_{I,\preq'}=S_{I,\preq}$, and  $\preq$ and $\preq'$ coincide on the initial section $S_{I,\preq}$, is a triangulation order, too.
\end{enumerate}
\end{rem}

We will prove the existence of triangulation orders for all facet ideals.  
The proof requires a case by case analysis. 
By Proposition \ref{face-ab-nil}, we may restrict the analysis to the abelian nilradicals.

\begin{defi}
We say that the facet ideal $I$ of $\Phi^+$  is an abelian nilradical of type $\mathrm X_{n,k}$, and we write $I\cong \mathrm X_{n,k}$, if there exists an irreducible root subsystem $\Psi$ of $\Phi$ and a positive system $\wt \Psi^+$ of $\Psi$ such that
$I$ is an abelian nilradical in $\wt \Psi^+$ and: 
\begin{enumerate}
\item 
$\Psi$ is of type $\mathrm X_n$; 
\item 
if $\{\al'_1, \dots, \al'_n\}$ is a simple system of $\wt \Psi^+$, numbered according to Bourbaki's conventions \cite{Bou}, then $I$ is the principal ideal generated by $\al'_k$ in $\wt \Psi^+$.
\end{enumerate}
\end{defi}

It is implicit in the definition that the above $\al'_k$ has multiplicity $1$ in $\Psi$. 

\par
We note that the type of a facet ideal may be not unique, if the root system $\Psi$ has nontrivial Dynkin diagram automorphisms. 
We identify the types $\mathrm X_{n,k}$ and $\mathrm X_{n,k'}$ if there exists a diagram automorphism that maps $\al_k$ into $\al_{k'}$.   
By a direct inspection of the highest root in all root types, we see that the possible types of abelian nilradicals type,  in an irreducible root system of rank $n$, are the following:  
$$
\mathrm A_{n, k}\ \text{ for } k=1,\dots, n, \quad \mathrm B_{n,1},\quad \mathrm C_{n, n},\quad \mathrm D_{n,k}\ \text{ for } k=1, n-1, n,\quad \mathrm E_{6,1},\quad  \mathrm E_{6,6}, \quad \mathrm E_{7, 7}.
$$
Among them, we have  the identifications: $\mathrm A_{n, k}=\mathrm A_{n, k'}$
for $k+k'=n+1$; $\mathrm D_{n,n-1}=\mathrm D_{n,n}$ for all $n\geq 4$ and  $\mathrm D_{n,1}=\mathrm D_{n,n-1}=\mathrm D_{n,n}$ for $n=4$; $\mathrm E_{6,1}=\mathrm E_{6,6}$.

\par
By  Proposition \ref{face-ab-nil},  the facet ideals that are not abelian nilradicals of $\Phi^+$ are in any case abelian nilradicals of some type. 
Their type  $\mathrm X_{n,k}$ is explicitly obtained as follows.

\par
Let $\al_i$ be a leaf in the extended Dynkin diagram of $\Phi$, so that $I_{\mathrm F_{\al_i}}$ is a facet ideal of $\Phi^+$. 
By Corollary \ref{cor-teo-facce}, the Dynkin diagram obtained by removing $\al_i$ from the extended Dynkin diagram of $\Phi$, gives the root type $\mathrm X_n$. 
The position of $-\theta$ in this diagram gives the index $k$ of the abelian nilradical type $\mathrm X_{n,k}$. 
Below, we write the resulting type for the facet ideals that are not abelian nilradicals of $\Phi^+$ itself. 
If the root type of $\Phi$ is  $\mathrm Y_n$, we write $I_F(\mathrm Y_n, \al_i)$ in place of $I_{F_{\al_i}}$.
\begin{align*}
&I_F(\mathrm B_n,\al_n)\cong \mathrm D_{n,n},\qquad 
I_F(\mathrm F_4,\al_4)\cong \mathrm B_{4,1},\qquad 
I_F(\mathrm E_7,\al_2) \cong \mathrm A_{7,1},
\\ 
&I_F({\mathrm E_8},\al_1)\cong \mathrm D_{8,1},\qquad I_F(\mathrm E_8,\al_2) \cong \mathrm A_{8,1}.
\end{align*}

\par
In proving next proposition, we will consider, case by case, the seven possible distinct sporadic or classes of abelian nilradical types.  
The main points of the proof are illustrated in Figures 1-9. 
We first give some explanation of these figures. 
We may arrange the roots of any facet ideal $I$ in a matrix $(\be_{i,j})$, in such a way that adjacent entries differ by a simple root. 
The label $i$ on a certain edge means that  the difference between its vertexes is the simple root $\al_i$. 
We choose the matrix arrangement of roots so that the standard partial order is compatible with the reverse lexicographic order of row and column indexes, starting from $\be_{1,1}=\theta$.
In this way, the matrix yields a Hasse diagram of $I$ in which  the order ascends toward northwest. 
We note that this condition do not determine a unique possibility. 
The figures illustrate  the proof on such Hasse diagrams for all the abelian nilradicals.

{\begin{pro}\label{triang-ord-exist}
Each facet ideal has a triangulation order.
\end{pro}

\begin{proof}
By the above discussion, we may assume that $I$ is an abelian nilradical of $\Phi^+$. 
\par
By Remark ~\ref{ini-sec}, it suffices to define a subset $S_{I,\preq}$  of $I$ and a partial order $\preq$ on $I$ that is total on $S_{I,\preq}$ and has $S_{I,\preq}$ as an initial section, in such a way that conditions (1) and (2) of Definition \ref{triang-order} are satisfied.
We also require $\rk(I\smeno S_{I,\preq})=n-1$, and $\rk(\be^\preq)=n$ for each $\be\in  S_{I,\preq}$, in order that $S_{I,\preq}=\{\be\in I\mid \rk(\be^\preq)=n\}$.
\par
Henceforward, we write $S_I$ in place of $S_{I,\preq}$ and we intend that $S_I$ is an initial section of $\preq$.
In all cases, we define the restriction $(S_I, \preq)$ as a sequence $(\be_1, \dots, \be_k)$ such that, for $i=1, \dots, k$,  $\be_i$ is an extremal element in $I\smeno \{\be_j\mid j< i\}$, with respect to the standard partial order. 
This ensures that $(\be_i^\preq)$ is saturated.
Therefore, in order to prove condition (2), it will remain to prove that either condition (a),  or (b) holds for all $\be_i$.
\par
If $\be_i$ is long and  $\be_i=\min(\be_i^\preq)$, or $\be_i=\max(\be_i^\preq)$ (with respect to the standard partial order), then $\be_i$ is detachable in  $(\be_i^\preq)$ by Lemma ~\ref{lem-detachable}, and we have nothing to prove.
In the remaining cases, we will directly prove that (a) or (b) holds.
\par

Finally, since  we take $\be_i$ extremal in $(\be_i^\preq)$  by construction, in order to prove that  $\be_i$ is detachable in  $(\be_i^\preq)$, or in a subset of its, it will suffice to check condition (2) in Definition~\ref{detachable}. 
\par

Now we can give the details of the proof for each abelian nilradical.
Throughout the rest of the proof, we use the following notation: for $h, k\in \{1, \dots, n\}$, $\omega_h=\check\omega_{\al_h}$,  $\al_{[h,k]}=\sum\limits_{h\leq i\leq k}\al_i$; for  $S\subseteq \{1, \dots, n\}$,  $\al_{S}=\sum\limits_{i\in S}\al_i$.
\smallskip

\nl 
{\bf A.} $I\cong \mathrm A_{n, k}$, $\left[\frac{n}{2}\right]< k\leq n$. 
We define $(S_I, \preq) =\(\al_{[k,j]}| j=k, \dots, n\)$.
It is easily seen that  $I\smeno S_I$ is the type $\mathrm A_{n-1, k-1}$ abelian nilradical generated by $\al_k+\al_{k-1}$ in the root subsystem that has
$\{\al_{k-1}+\al_k\}\cup \Pi\smeno\{\al_{k-1},\al_k\}$ as a simple system. 
This implies that $\rk(I\smeno S_I)=n-1$ and, by Lemma \ref{lemma-simclosed}, that $I\cap \gen (I\smeno S_I)$ is $\sim$closed.
\par 
For all $\be\in S_I$, $\rk(\be^\preq)=n$ and $\be$ is detachable in  $(\be^\preq)$. 
Indeed, for $\be=\al_{[k,j]}$, with $j\in[k,n]$, let  $H=(\check\omega_k-\check\omega_{k-1}-\check\omega_{j+1})^\perp$ (where $\check\omega_{n+1}=\0$). 
Then, if $j< n$,  the simple system of $(\Phi\cap H)^+$ is $\left\{\al_{[k-1,k]}, \al_{[k, j+1]}\right\}\cup \Pi\smeno\left\{\al_{k-1}, \al_k,\al_{j+1}\right\}$, while the maximal roots are $\al_{[1, j]}$ and $\al_{[k, n]}$.
If $j= n$, the simple system  is $\left\{\al_{[k-1,k]}\right\}\cup \Pi\smeno\left\{\al_{k-1}, \al_k\right\}$, and $\Phi\cap H$ is irreducible.
By Lemma  \ref{lemma-simclosed},  we obtain that $I\cap H$ is $\sim$closed.
It remains to check that condition (2a) of Definition ~\ref{detachable} hold. 
Let $\ga\in (\be^\preq)$.
If $\ga\in H$, either $\ga$ and $\be$ are incomparable for the standard partial order,  or $\ga-\be\in \Phi^+$, while, if $\ga\not\in H$, we have  $\ga\geq \be$ and $(\ga, \be)=0$.
By Proposition \ref{minori-crossing}, we obtain that $\ga\sim \be$ if and only if $\ga\not\in H$, which is the claim.
\smallskip

\nl {\bf C.} $I\cong\mathrm C_{n, n}$. We define $(S_I, \preq)=(\al_{[j,n]}| j=n,n-1, \dots, 1)$. 
It is easy to see that $I\smeno S_I$ is the type $\mathrm C_{n-1, n-1}$ abelian nilradical generated by $\al_{n}+2\al_{n-1}$ in the root subsystem that has  $\{\al_{n}+2\al_{n-1}\}\cup \Pi\smeno\{\al_{n-1},\al_n\}$ as a simple system.
Hence, $\rk(I\smeno S_I)=n-1$ and, by Lemma \ref{lemma-simclosed},  $I\cap \gen (I\smeno S_I)$ is $\sim$closed.
\par
For all $\be\in S_I$, $\rk(\be^\preq)=n$ and $\be$ is detachable in  $(\be^\preq)$. 
Indeed, for $\be=\al_{[j,n]}$, $j\in[1,n]$, we take  $H=(2\check\omega_n-\check\omega_{n-1}-\check\omega_{j-1})^\perp$. 
Then, the simple system of $(\Phi\cap H)^+$ is $\left\{\al_{[j-1,n]}, \al_n+2\al_{n-1}\right\}\cup \Pi\smeno\left\{\al_n, \al_{n-1}, \al_{j-1}\right\}$ for $j<n$, and $\left\{\al_n+\al_{n-1}\right\}\cup \Pi\smeno\left\{\al_n, \al_{n-1}\right\}$ for $j=n$. 
For $j<n$, the maximal roots of $(\Phi\cap H)^+$ are $\al_{[j,n]}+\al_{[j,n-1]}$ and $\al_{[1,n]}$. 
For $j=n$, $\Phi\cap H$ is irreducible.
It follows that $I\cap H$ is $\sim$ closed, by Lemma~\ref{lemma-simclosed}.
If $\ga\in I$,  then $\ga=\al_{[h,n]}+\al_{[k, n-1]}$ for some $1\leq h\leq k\leq n$. 
Hence, $\ga\in H$ if and only if either $h\leq j-1$ and $k=n$, or $j\leq h\leq k\leq n-1$. 
In these cases, either $\ga$ and $\be$ are incomparable for the standard partial order,  or $\ga-\be\in \Phi^+$, and all $\ga'$ such that  $\ga<\ga'<\be$ are short roots. 
In any case, $\ga\nosim\be$ by Proposition \ref{minori-crossing}(3). 
If $\ga\in (\be^\preq)\smeno H$, we have  $\ga=\al_{[h,n]}+\al_{[k, n-1]}$ with $h\leq j-1\leq k\leq n-1$, hence $\be+\al_{[k,n-1]}\in\Phi$ and we obtain a crossing relation.
It follows that $H$ satisfies the conditions of Definition \ref{detachable}. 
\smallskip

\nl {\bf B} and  {\bf D1.} $I\cong \mathrm B_{n, 1}$, or $I\cong \mathrm D_{n, 1}$. 
We define $(S_I, \preq)=(\al_1, \theta)$.
It is easy to see that $I\smeno S_I$ is the type  $\mathrm B_{n-1, 1}$, or $\mathrm D_{n-1, 1}$, abelian nilradical generated by $\al_1+\al_2$ in the subsystem whose simple system is $\{\al_1+\al_2\}\cup \Pi\smeno\{\al_1,\al_2\}$.
Hence, $\rk(I\smeno S_I)=n-1$ and, by Lemma \ref{lemma-simclosed},  $I\cap \gen (I\smeno S_I)$ is $\sim$closed.
For all $\be\in S_I$, $\rk(\be^\preq)=n$ and $\be$ is detachable in  $(\be^\preq)$ by Lemma ~\ref{lem-detachable}.
\smallskip

\nl {\bf Dn.} $I\cong\mathrm D_{n, n}$. 
We define $(S_I, \preq)=(\wh\al_{[j,n]}| j=n,n-2, \dots, 1)$, where $\wh\al_{[j,n]}:=\al_n+\al_{[j, n-2]}$.
It is easy to see that $I\smeno S_I$ is the type $\mathrm D_{n-1,n-1}$ abelian nilradical generated by $\al_{[n-2,n]}$ in the root subsystem that has  $\{\al_{[n-2,n]}\}\cup \Pi\smeno\{\al_{n-1},\al_n\}$ as a simple system.
Hence, $\rk(I\smeno S_I)=n-1$ and, by Lemma \ref{lemma-simclosed},  $I\cap \gen (I\smeno S_I)$ is $\sim$closed.
\par
It remains to prove that all $\be\in S_I$, $\rk(\be^\preq)=n$ are  detachable in  $(\be^\preq)$. 
If $\be=\al_n$ or $\be=\al_{n}+\al_{n-2}$, then $\be$ is detachable in  $(\be^\preq)$ by  Lemma ~\ref{lem-detachable}. 
Otherwise, let $\be=\wh\al_{[j,n]}$, $j\in\{1,\dots,n-3\}$.
In this case we have a bipartition $(\be^\preq)=J_{\mathrm i}\cup J_{\mathrm f}$ with $J_{\mathrm f}=(\be^\leq) $ and $J_{\mathrm i}=(\be^\preq)\smeno (\theta_k^\leq) $. 
Indeed, we have: $(\be^\preq)=\{\ga\in I\mid c_{\al_j}(\ga)\geq 1 \text { or }  c_{\al_{n-1}}(\ga)\geq 1\}$, $J_{\mathrm f}\smeno J_{\mathrm i}=\{\ga\in I\mid c_{\al_j}(\ga)=2\}$, and $J_{\mathrm i}\smeno J_{\mathrm f}=\{\ga\in I\mid c_{\al_j}(\ga)=0 \text{ and } c_{\al_{n-1}}(\ga)=1\}$.
Hence, if we set $H=(\check\omega_n-\check\omega_j)^\perp$, $H$ strictly separates $J_{\mathrm i}\smeno J_{\mathrm f}$ from $J_{\mathrm f}\smeno J_{\mathrm i}$, and $J_{\mathrm i}\cap J_{\mathrm f}=I\cap H$. 
Moreover, for any $\ga_1\in J_{\mathrm i}\smeno J_{\mathrm f}$ and $\ga_2\in J_{\mathrm f}\smeno J_{\mathrm i}$ we have $c_{\al_{n-1}}(\ga_2-\ga_1)=0$ and $c_{\al_j}(\ga_2-\ga_1)=0$, hence $\ga_2-\ga_1\not\in \Phi$ and $\ga_1\lesssim \ga_2$ by Proposition ~\ref{minori-crossing}. 
By Lemma~\ref{lem-detachable}, $\be$ is detachable in  $J_{\mathrm f}$ and there exists a detaching hyperplane $H^{\mathrm f}$ for $\be$ in $J_{\mathrm f}$
such that $H^{\mathrm f}\cup (\be^\preq)$ is contained in $\red(\be)$.
The proof that $\be$ is also detachable in  $J_{\mathrm i}$ will be very similar to the proof of cases~${\mathrm A}_{n}$.
We take  $H^{\mathrm i} =(\check\omega_n-\check\omega_{n-1}-\check\omega_{j-1})^\perp$. 
Then, for each $\ga\in(\be^\prec)$, if $\ga\notin H^{\mathrm i}$, we have $\ga> \be$ and $(\ga, \be)=0$.
If $\ga\in H^{\mathrm i}$, either $\ga$ is incomparable with $\be$, or $\ga-\be\in \Phi$. 
Hence, by Proposition \ref{minori-crossing},  $\ga\sim\be$ if and only if $\ga\not\in H^{\mathrm i}$. 
The simple system for $(\Phi\cap H^{\mathrm i})^+$ is $\left\{\al_{[n-1,n]}, \wh\al_{[j-1,n]}\right\}\cup \Pi\smeno\left\{\al_{j-1}, \al_{n-1}, \al_n\right\}$,
for $j> 1$ and $\left\{\al_{[n-1,n]}\right\}\cup \Pi\smeno\left\{\al_{n-1}, \al_n\right\}$ for $j=1$.
For $j>1$, the maximal roots are $\wh\al_{[1,n]}$ and $\al_{[j,n]}+\al_{[j+1, n-2]}$, while,  for $j=1$, $\Phi\cap H$ is irreducible.
Hence $I\cap H^{\mathrm i}$ is $\sim$closed  by Lemma \ref{lemma-simclosed}.
\smallskip

\nl {\bf E6.} $I\cong \mathrm E_{6, 6}$. 
We choose  $(S_I,\preq)=\(\al_6, \theta, \al_{\{5,6\}}, \theta-\al_2, \al_{\{4,5,6\}}, \theta-\al_{\{2,4\}}, \al_{\{2,4,5,6\}},\right.$
$\left.\theta-\al_{\{2, 4, 5\}}\).$
Then $I\smeno S_I$ is the type $\mathrm A_{5,2}$ abelian nilradical generated by $\al_{[3,6]}$ in the root subsystem with simple system $\{\al_{[3,6]}\}\cup \Pi\smeno\{\al_3, \al_6\}$. 
Hence $I\cap\gen(I\smeno S_I)$ is $\sim$closed by Lemma \ref{lemma-simclosed}. 
The first six $\be$ in $(S_I,\preq)$ are detachable in  their $(\be^\preq)$ by Lemma ~\ref{lem-detachable}. 
Hence we have to consider only the last two roots. 
These are symmetric with respect to the order involution of $I$, hence it suffices to consider $\be=\al_{\{2,4,5,6\}}$.
By Proposition \ref{minori-crossing},  $\be\not\sim\ga$ for all $\ga\in (\be^\preq)$ except $\ga=\theta-\al_{\{2, 4, 5\}}$. 
Then, the hyperplane $H=(\check\omega_6-\check\omega_3)^\perp$ strictly separates $\be$ from $\theta-\al_{\{2, 4, 5\}}$ and contains all other roots in $(\be^\preq)$. 
The simple system of $(\Phi\cap H)^+$ is $\left\{\al_2, \dots, \al_5\right\}\cup\left\{\al_{[3,6]}\right\}$, and $(\Phi\cap H)$ is irreducible. Hence, we may apply Lemma \ref{lemma-simclosed} and conclude that $\be$ is detachable in  $(\be^\preq)$. 
\smallskip

\nl {\bf E7.} $I\cong \mathrm E_{7, 7}$. 
We choose  $(S_I, \preq)=
\left(\al_7, \theta, \al_{\{6, 7\}}, \theta-\al_1, 
 \al_{\{5, 6,7\}}, \theta-\al_{\{1,3\}}, \al_{\{4, 5, 6,7\}},\right.$ $\left.\theta-\al_{\{1,3,4\}},\al_{\{2, 4, 5, 6,7\}}, \theta-\al_{\{1,3,4,2\}}, \al_{\{3, 4, 5, 6,7\}}, \theta-\al_{\{1,3,4,5\}}, \al_{\{1, 3, 4, 5, 6,7\}}, \theta-\al_{\{1,3,4,5,6\}}\right)$.
We note that $(S_I,\preq)$ consists of all $\be$ in $I$ such that $c_{\al_i}(\be)\leq 1$ for $i=1,\dots, 7$, together with their symmetric roots, with respect to the order involution of $I$.
\par
If  $\be=\al_{\{2, 4, 5, 6,7\}}$ and $\be'=\theta-\al_{\{1,3,4,2\}}$ (as in Figure 7), then  $I\cap \gen(I\smeno S_I)=(I\smeno S_I)\cup \{\be, \be'\}$.
This is the type $\mathrm D_{6,6}$ abelian nilradical generated by $\be$ in the subsystem that has $\{\be\}\cup \Pi\smeno\{\al_2, \al_7\}$.
Hence $(I\smeno S_I)\cup \{\be, \be'\}$ is $\sim$closed by Lemma \ref{lemma-simclosed}.
\par
All roots in $S_I$, except $\al_{\{2, 4, 5, 6,7\}}$, $\al_{\{1, 3, 4, 5, 6,7\}}$, and their symmetric roots with respect to the order involution, are detachable in  their $\preq$-upper cone by Lemma ~\ref{lem-detachable}. 
\par
Let $\be=\al_{\{2, 4, 5, 6,7\}}$. 
Then $(\be^\preq)$ has the bipartition $J_{\mathrm i}\cup J_{\mathrm f}$ with $J_{\mathrm f}=(\be^\preq)\cap (\be^\leq) =\{\ga\in (\be^\preq)\mid c_{\al_2}(\ga)\geq 1\}$ and $J_{\mathrm i}=\{\ga\in (\be^\preq)\mid c_{\al_2}(\ga)\leq 1\}$. 
Indeed, it is easily seen that $H=(\check\omega_7-\check\omega_2)^\perp$ contains $J_{\mathrm i}\cap J_{\mathrm f}$ and strictly separates $J_{\mathrm i}\smeno J_{\mathrm f}$ from $J_{\mathrm f}\smeno J_{\mathrm i}$. 
Moreover, it is easy to see that,  for all $\ga_1\in J_{\mathrm i}\smeno J_{\mathrm f}$ and $\ga_2\in J_{\mathrm f}\smeno J_{\mathrm i}$, $\ga_2-\ga_1 \in 
L^+(\Phi)\smeno \Phi^+$, hence $\ga_1\lesssim \ga_2$. 
It remains to check that $\be$ is detachable in  $J_{\mathrm i}$ and $J_{\mathrm f}$ and the further conditions in Definition \ref{triang-order} (2b) hold.
For $J_{\mathrm f}$ this follows from Lemma~\ref{lem-detachable}. 
For $J_{\mathrm i}$, the hyperplane  $H^{\mathrm i}:=(\check\omega_7-\check\omega_3)$ is a detaching hyperplane that satisfies the required conditions. 
Indeed, the simple system of $(\Phi\cap H^{\mathrm i})$ is $\{\al_{\{3,4,5,6,7\}}\}\cup \Pi\smeno\{\al_7, \al_3\}$, hence $I\cap H^{\mathrm i}$ is $\sim$closed by Lemma \ref{lemma-simclosed}. 
Moreover, we can easily check that for  $\ga\in J_{\mathrm i} \cap H^{\mathrm i}$, either $\ga$ and $\be$ are incomparable, or $\ga\smeno \be\in \Phi^+$, while for all $\ga\in J_{\mathrm i} \smeno H^{\mathrm i}$, we have $\ga> \be$ and $\ga-\be\not\in \Phi$. 
Hence, we may conclude that  by Lemma \ref{minori-crossing}.
\par
Let $\be=\al_{\{1, 3, 4, 5, 6,7\}}$. 
In this case the hyperplane $\be$ is detachable in  $(\be^\preq)$, since $H:=(\check\omega_7-\check\omega_2)^\perp$ satisfies the conditions of Definition ~\ref{detachable}.
The details are similar to the previous case.
\par
Finally, for $\be=\theta-\al_{[1,4]}$ and $\theta-\al_{\{1,2,4,5,6\}}$ we have similar results by symmetry.
\end{proof}

\begin{figure}[hpt]
\caption{$I\cong \mathrm A_{9, 6}$. Here $\be=\al_{[6
, 7]}$. The gray boxes cover the roots in $H:=(\check\omega_6-\check\omega_5-\check\omega_8)^\perp$.}
$$
\begin{tikzpicture}[scale=1.2]
\foreach\x in{1,...,6} \draw(\x,6)--(\x,9); 
\foreach\y in{6,...,9} \draw(1,\y)--(6, \y);
\foreach\y in{6,...,9} \foreach\x in{1,...,6} 
\filldraw (\x,\y) circle(0.05cm);
\foreach\y in {6,...,9}
\draw (6,\y) circle(0.15cm);
\foreach\x in{1,...,6} \foreach\y in{7,...,9}
\node[left] at (\x+0.05, \y-.5) {$\ssty\y$};
\foreach\x in{1,...,5} \foreach\y in{6,...,9}
\node[above] at (\x+.5, \y-.05) {$\ssty\x$};
\node[right] at (6.2,7) {$\be$}; 
\filldraw[gray,opacity=0.2] (.7,7.3) rectangle (5.3, 5.7);
\filldraw[gray,opacity=0.2] (5.7,9.3) rectangle (6.3, 7.7);
\node[left] at (1,9) {$\theta$}; 
\node[right] at (6.2,6) {$\al_6$}; 
\end{tikzpicture}
$$
\end{figure}
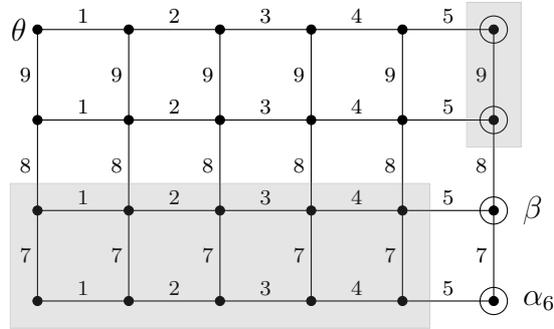

\begin{figure}[hpt]
\caption{$I\cong \mathrm C_{7, 7}$. Here $\be=\al_{[4, 7]}$. 
The gray boxes cover the roots in $H:=(2\check\omega_7-\check\omega_3-\check\omega_6)^\perp$.}
$$
\begin{tikzpicture}[scale=1.2]
\foreach\x in{1,...,7} \draw(\x,-1)--(\x,-\x); 
\foreach\y in{-1,...,-7} \draw(7, \y)--(-\y,\y);
\foreach\x in {1,...,7}
\foreach\y in {-\x, ..., -1}
\filldraw (\x,\y) circle(0.05cm);
\foreach\x in {1,...,6}
\foreach\y in {-1, ..., -\x}
\node[above] at (\x+.5,\y-.05) {$\ssty\x$};
\foreach\y in {1,...,6}
\foreach\x in {\y,...,6}
\node[right] at (\x+0.95,-\y-.5) {$\ssty\y$};
\foreach\y in {-1,..., -7}
\draw (7,\y) circle(0.15cm);
\filldraw[gray,opacity=0.2] (3.7,-3.7) rectangle (6.3, -6.3);
\filldraw[gray,opacity=0.2] (6.7,-.7) rectangle (7.3, -3.3);
\node[left] at (1,-1) {$\theta$}; 
\node[right] at (7.2,-7) {$\al_7$}; 
\node[right] at (7.2,-4) {$\be$}; 
\end{tikzpicture}
$$
\end{figure}
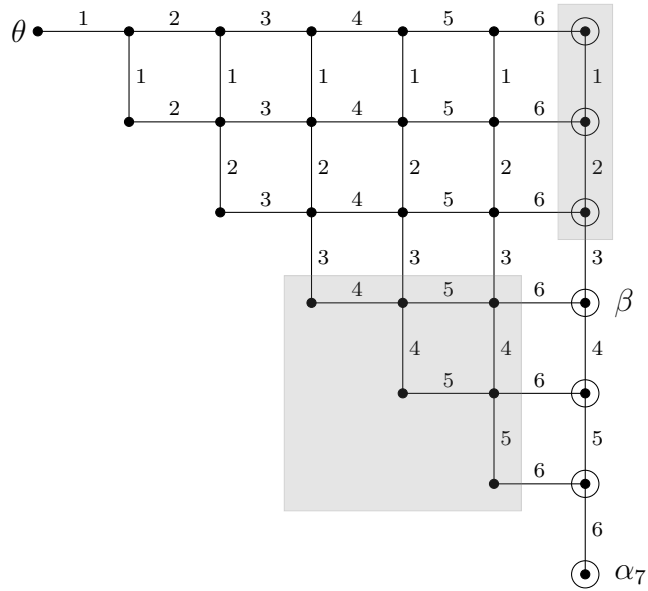

\begin{figure}[hpt]
\caption{$I\cong \mathrm B_{6, 1}$ and $I\cong \mathrm D_{6, 1}$. In both cases $S_I=\{\al_1, \theta\}$. 
The gray boxes cover the roots in $H=(\check\omega_1-\check\omega_2)^\perp$, for either  $\be=\al_1$ and $\be=\theta$.} 
$$
\begin{tikzpicture}[scale=1.2]
\draw(-5,0)--(5,0); 
\foreach\x in{-5,...,5} \filldraw (\x,0) circle(0.05cm);
\foreach\x in{2,...,6} 
\node[above] at (\x-6.5, -0.05) {$\ssty \x$};
\foreach\x in{2,...,6} 
\node[above] at (-\x+6.5, -0.05) {$\ssty \x$};
\node[right] at (5.2,0){$\al_1$};
\node[below] at (0,0){$\theta_s$};
\node[left] at (-5.2,0){$\theta$};
\foreach\x in{-5, 5}\draw (\x,0) circle(0.15cm);
\filldraw[gray, opacity=0.2] (-4.2, -.4) rectangle (4.2, .4);
\end{tikzpicture}
$$

$$
\begin{tikzpicture}[scale=1.2]
\draw(-4,0)--(0,0); 
\draw(-1,-1)--(3,-1); 
\draw(-1,0)--(-1,-1); 
\draw(0,0)--(0,-1); 
\foreach\x in{-4,...,0} \filldraw (\x,0) circle(0.05cm);
\foreach\x in{-1,...,3} \filldraw (\x,-1) circle(0.05cm);
 \filldraw (-1,-1) circle(0.05cm);
\foreach\x in {2,...,5}\node[above] at (\x-5.5, -0.05) {$\ssty \x$};
\foreach\x in {2,...,5}\node[above] at (-\x+4.5, -1.05) {$\ssty \x$};
\foreach\x in{-1,...,0} \node[right] at (\x-0.05, -0.5) {$\ssty 6$};
\node[right] at (3.2,-1){$\al_1$};
\node[left] at (-4.2,0){$\theta$};
\draw (-4,0) circle(0.15cm);
\draw (3,-1) circle(0.15cm);
\filldraw[gray, opacity=0.2] (-3.2, .4) rectangle (2.2, -1.4);
\end{tikzpicture}
$$
\end{figure}
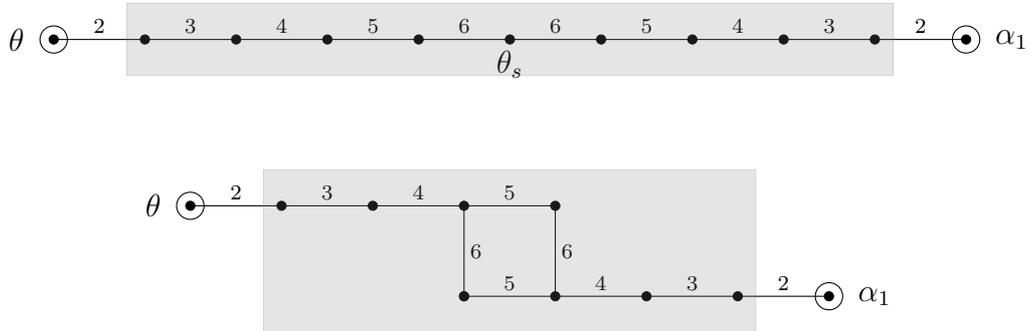

\begin{figure}[hpt]
\caption{$I\cong \mathrm D_{8,8}$. The figure represents the Hasse diagram of the whole $I$. The gray boxes illustrate the bipartition of $(\be^\preq)$ described in the proof, for $\be=\widehat\al_{4, 8}$. The next figure represents~$(\be^\preq)$.} 
$$
\begin{tikzpicture}[scale=1.2]
\foreach\x in{0,...,6}  \draw(\x, -\x)--(6, -\x); 
\foreach\x in{0,...,6}  \draw(\x, -\x)--(\x, 0); 
\foreach\x in{0,...,6} \foreach\y in{-\x, ..., 0} 
\filldraw (\x,\y) circle(0.05cm);
\foreach\y in {1, ...,6} \foreach\x in{\y,...,6} 
\node[left]at (\x+.1, -\y+0.5) {$\ssty \y$};
\foreach\x in {2, ...,7} \foreach\y in{-\x,...,-2} 
\node[above] at (\x- 1.5, \y+1.9) {$\ssty \x$};
\node[below] at (6,-6.1){$\al_8$};
\node[left] at (0,0){$\theta$};
\node[right] at (6.1,-3) {$\be$}; 
\foreach\y in{-6, ..., 0}\draw (6,\y) circle(0.15cm);
\filldraw[gray, opacity=0.2] (-.4,.3) rectangle (2.3, -2.2);
\node[left] at (-0.3,-1) {$J_{\mathrm f}\smeno J_{\mathrm i}$}; 
\filldraw[gray, opacity=0.2] (3.7,-3.7) rectangle (5.3, -5.2);
\node[left] at (3.7,-4.5) {$J_{\mathrm i}\smeno J_{\mathrm f}$}; 
\filldraw[gray,opacity=0.2] (2.7,.5) rectangle (6.6, -3.3);
\node[right] at (6.5,-1.5) {$J_{\mathrm i}\cap J_{\mathrm f}$}; 
\end{tikzpicture}
$$
\end{figure}
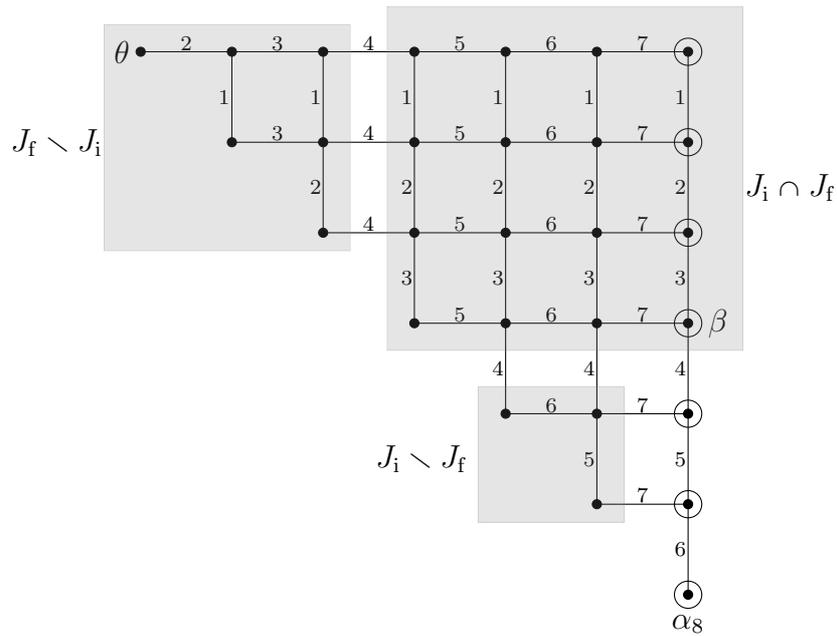

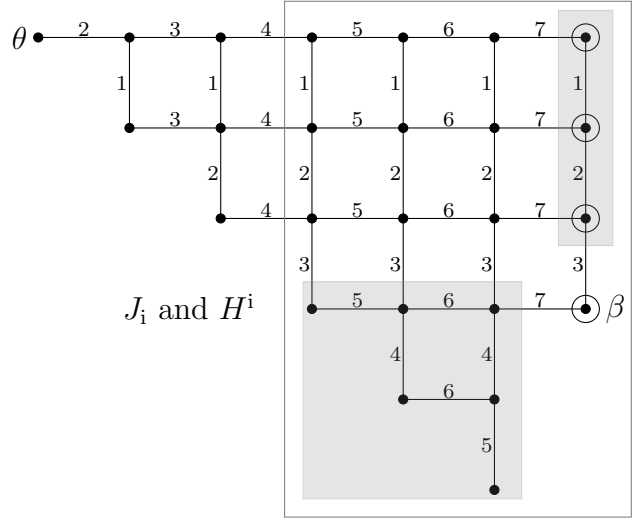
\begin{figure}[hpt]
\caption{$I\cong \mathrm D_{8,8}$. The diagram represents $(\be^\preq)$ for $\be=\widehat\al_{4, 8}$. The big rectangle contains the roots in $J_{\mathrm i}$
 and the gray parts cover the roots in $H^{\mathrm i} =(\check\omega_8-\check\omega_7-\check\omega_3)^\perp$.}
$$
\begin{tikzpicture}[scale=1.2]
\foreach\x in{0,...,3}  \draw(\x, -\x)--(6, -\x); 
\foreach\x in{4,...,5}  \draw(\x, -\x)--(5, -\x); 
\foreach\x in{0,...,5}  \draw(\x, -\x)--(\x, 0); 
\draw(6, -3)--(6, 0); 
\foreach\x in{0,...,5} \foreach\y in{-\x, ..., 0} 
\filldraw (\x,\y) circle(0.05cm);
\foreach\y in{-3, ..., 0}
\filldraw (6,\y) circle(0.05cm);
\foreach\y in {1, ...,5} \foreach\x in{\y,...,5} 
\node[left]at (\x+.1, -\y+0.5) {$\ssty \y$};
\foreach\y in {1, ...,3} 
\node[left]at (6.1, -\y+0.5) {$\ssty \y$};
\foreach\x in {2, ...,6} \foreach\y in{-\x,...,-2} 
\node[above] at (\x- 1.5, \y+1.9) {$\ssty \x$};
\foreach\y in{-3,...,0} 
\node[above] at (5.5, \y-.1) {$\ssty 7$};
\node[left] at (0,0){$\theta$};
\node[right] at (6.1,-3) {$\be$}; 
\foreach\y in{-3, ..., 0}\draw (6,\y) circle(0.15cm);
\draw[gray] (2.7,.4) rectangle (6.5, -5.3);
\filldraw[gray, opacity=0.2] (2.9,-2.7) rectangle (5.3, -5.1);
\filldraw[gray, opacity=0.2] (5.7,.3) rectangle (6.3, -2.3);
\node[left] at (2.5,-3) {$J_{\mathrm i}$ and $H^{\mathrm i}$}; 
\end{tikzpicture}
$$
\end{figure}

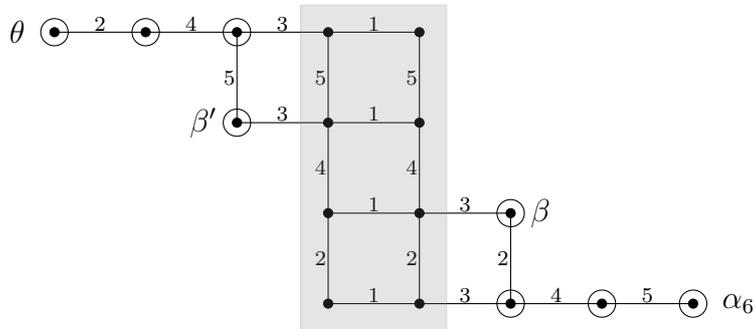
\begin{figure}[hpt]\caption{$I\cong \mathrm E_{6,6}$. 
The gray rectangle covers the roots in  $H=H_{\be'}=(\check\omega_6-\check\omega_3)^\perp$ for $\be=\al_{\{2,4,5,6\}}$ and $\be'=\theta-\al_{\{2,4,5\}}$. 
By definition $(\be^\preq)$ consists of these roots plus $\be$ and $\be'$, while $(\be'^\preq)$ consists of the roots in gray rectangle plus  $\be'$.}

$$
\begin{tikzpicture}[scale=1.2]
\draw (0,0)--(0,3);
\draw (0,1)--(2,1);
\draw (-1,2)--(1,2);
\draw (-3,3)--(1,3);
\draw (-1,2)--(-1,3);
\draw (0,0)--(4,0);
\draw (1,0)--(1,3);
\draw (2,0)--(2,1);
\foreach\x in {0, ..., 4}
\filldraw(\x,0) circle(0.05cm);
\foreach\x in {0, ..., 2}
\filldraw(\x,1) circle(0.05cm);
\foreach\x in {-1, ..., 1}
\filldraw(\x,2) circle(0.05cm);
\foreach\x in {-3, ..., 1}
\filldraw(\x,3) circle(0.05cm);
\node[above] at (-2.5,2.9)  {$\ssty 2$};
\node[above] at (-1.5,2.9)  {$\ssty 4$};
\foreach\y in {2,...,3}\node[above] at (-0.5,\y-.1) {$\ssty 3$};
\foreach\y in {0,...,3}\node[above] at (0.5,\y-.1) {$\ssty 1$};
\foreach\y in {0,...,1}\node[above] at (1.5,\y-.1) {$\ssty 3$};
\node[above] at (2.5,-.1)  {$\ssty 4$};
\node[above] at (3.5,-.1)  {$\ssty 5$};
\foreach\x in {-1,...,1}\node[left] at (\x+.1,2.5) {$\ssty 5$};
\foreach\x in {0,...,1}\node[left] at (\x+.1,1.5) {$\ssty 4$};
\foreach\x in {0,...,2}\node[left] at (\x+.1,0.5) {$\ssty 2$};
\node[left] at (-3.2,3)  {$\theta$};
\node[right] at (4.2,0)  {$\al_6$};
\foreach \x in {2,3,4}\draw (\x,0) circle(0.15cm); 
\foreach \x in {-3,-2,-1}\draw (\x,3) circle(0.15cm); 
\draw (-1,2) circle(0.15cm); 
\draw (2,1) circle(0.15cm); 
\filldraw[gray, opacity=0.2](-.3,3.3) rectangle (1.3,-.3);
\node[right] at (2.1, 1) {$\be$};
\node[left] at (-1.1, 2) {$\be'$};
\end{tikzpicture}
$$
\end{figure}

\begin{figure}[hpt]\caption
{$I\cong \mathrm E_{7, 7}$. The figure represents the full Hasse diagram of $I$. The gray rectangles illustrate the bipartition of $(\be^\preq)$ for $\be=\al_{\{2,4,5,6,7,\}}$. 
The bipartition of $(\be'^\preq)$, for the symmetric root $\be'=\theta-\al_{[1,4]}$,  is similar.} 
$$
\begin{tikzpicture}[scale=1.1]
\draw (-4,4)--(-4,-1);
\draw (-1,-4)--(4,-4);
\draw (-4,1)--(-3, 1)--(-3, -1);
\foreach\y in{0, -1}\draw (-4,\y)--(0,\y);
\foreach\x in{0, -1}\draw (\x,0)--(\x,-4);
\draw(-2, 0)--(-2, -2)--(0, -2);
\draw (-1,-3)--(1,-3)--(1, -4);
\node[left] at (-3.9,3.5)  {$\ssty 1$};
\node[left] at (-3.9, 2.5)  {$\ssty 3$};
\node[left] at (-3.9, 1.5)  {$\ssty 4$};
\foreach\x in {-4,-3}\node[left] at (\x+.1,0.5)  {$\ssty 5$};
\foreach\x in {-4,...,0}\node[left] at (\x+.1,-0.5)  {$\ssty 6$};
\foreach\x in {-2,..., 0}\node[left] at (\x+.1,-1.5)  {$\ssty 5$};
\foreach\x in {-1,0}\node[left] at (\x+.1,-2.5)  {$\ssty 4$};
\foreach\x in {-1,..., 1}\node[left] at (\x+.1,-3.5)  {$\ssty 2$};
\node[above] at (3.5,-4.1)  {$\ssty 6$};
\node[above] at (2.5,-4.1)  {$\ssty 5$};
\node[above] at (1.5,-4.1)  {$\ssty 4$};
\foreach\y in {-4,-3}\node[above] at (0.5,\y-.1)  {$\ssty 3$};
\foreach\y in {-4,...,0}\node[above] at (-0.5,\y-.1)  {$\ssty 1$};
\foreach\y in {-2,..., 0}\node[above] at (-1.5,\y-.1)  {$\ssty 3$};
\foreach\y in {-1,0}\node[above] at (-2.5,\y-.1)  {$\ssty 4$};
\foreach\y in {-1,..., 1}\node[above] at (-3.5,\y-.1)  {$\ssty 2$};
\foreach\y in {-1,...,4}
\filldraw(-4,\y) circle(0.05cm);
\foreach\x in {-1,...,4}
\filldraw(\x,-4) circle(0.05cm);
\foreach\x in {-1,...,1}
\filldraw(\x, -3) circle(0.05cm);
\foreach\y in {-1,...,1}
\filldraw(-3, \y) circle(0.05cm);
\foreach\x in {-2,...,0}\foreach\y in {-2,...,0}
\filldraw(\x, \y) circle(0.05cm);
\node[above] at (-4,4.2) {$\theta$}; 
\node[right] at (4.2,-4) {$\al_7$}; 
\foreach\y in {-1, ..., 4} 
\draw(-4, \y) circle (0.15cm); 
\foreach\x in {-1, ..., 4} 
\draw(\x, -4) circle (0.15cm); 
\draw(1, -3) circle (0.15cm); 
\draw(-3,1) circle (0.15cm); 
\node[right] at (1.1,-3){$\be$};
\node[above] at (-3,1.05){$\be'$};
\filldraw[gray, opacity=0.2] (-3.3, 1.6)rectangle(1.5, -3.3); 
\node[right]at(1.5,0){$J_{\mathrm i}\cap J_{\mathrm f}$};
\filldraw[gray, opacity=0.2] (-4.3, .3)rectangle(-3.7, -1.3); 
\node[left]at(-4.5,-.5){$J_{\mathrm f}\smeno J_{\mathrm i}$};
\filldraw[gray, opacity=0.2] (.3,-4.3)rectangle(-1.3,-3.7); 
\node[below]at(-.5,-4.5){$J_{\mathrm i}\smeno J_{\mathrm f}$};
\end{tikzpicture}
$$
\end{figure}
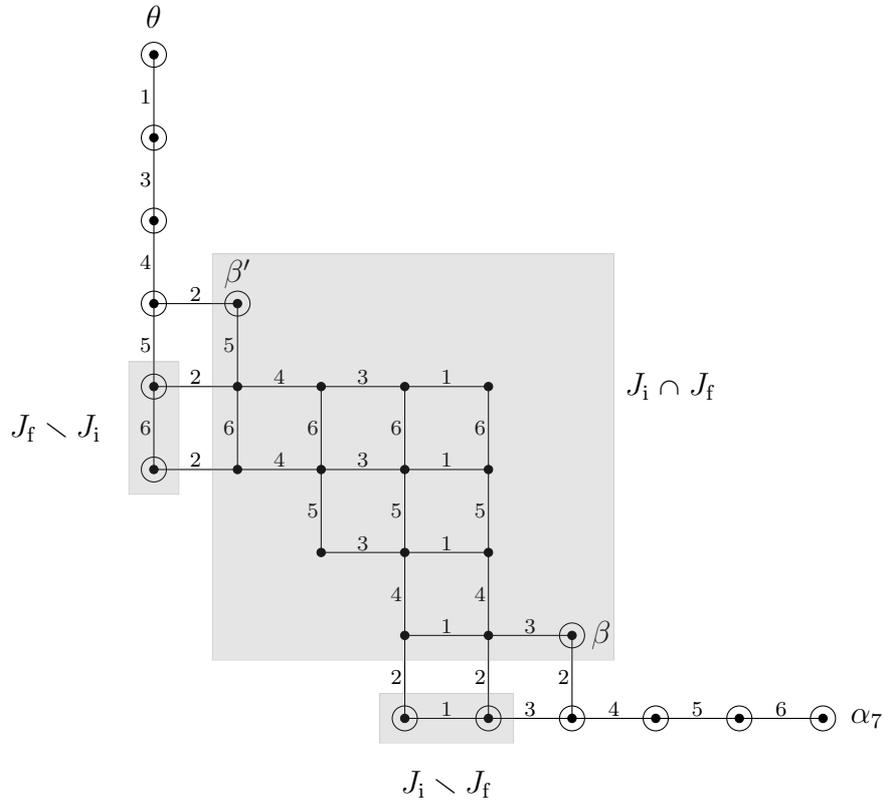

\begin{figure}[hpt]\caption
{$I\cong \mathrm E_{7, 7}$. The diagram represents  $(\be^\preq)$ for $\be=\al_{\{2,4,5,6,7,\}}$.  
The big rectangle contains the roots in $J_{\mathrm i}$ and gray part covers the roots in $H^{\mathrm i}=(\check\omega_7-\check\omega_3)^\perp$.} 

$$
\begin{tikzpicture}[scale=1.2]
\draw (-4,0)--(-4,-1);
\draw (-1,-4)--(0,-4);
\draw (-3, 1)--(-3, -1);
\foreach\y in{0, -1}\draw (-4,\y)--(0,\y);
\foreach\x in{0, -1}\draw (\x,0)--(\x,-4);
\draw(-2, 0)--(-2, -2)--(0, -2);
\draw (-1,-3)--(1,-3);
\foreach\x in {-3}\node[left] at (\x+.1,0.5)  {$\ssty 5$};
\foreach\x in {-4,...,0}\node[left] at (\x+.1,-0.5)  {$\ssty 6$};
\foreach\x in {-2,..., 0}\node[left] at (\x+.1,-1.5)  {$\ssty 5$};
\foreach\x in {-1,0}\node[left] at (\x+.1,-2.5)  {$\ssty 4$};
\foreach\x in {-1,...,0}\node[left] at (\x+.1,-3.5)  {$\ssty 2$};
\foreach\y in {-3}\node[above] at (0.5,\y-.1)  {$\ssty 3$};
\foreach\y in {-4,...,0}\node[above] at (-0.5,\y-.1)  {$\ssty 1$};
\foreach\y in {-2,..., 0}\node[above] at (-1.5,\y-.1)  {$\ssty 3$};
\foreach\y in {-1,0}\node[above] at (-2.5,\y-.1)  {$\ssty 4$};
\foreach\y in {-1,..., 0}\node[above] at (-3.5,\y-.1)  {$\ssty 2$};
\foreach\y in {-1,...,0}
\filldraw(-4,\y) circle(0.05cm);
\foreach\x in {-1,...,0}
\filldraw(\x,-4) circle(0.05cm);
\foreach\x in {-1,...,1}
\filldraw(\x, -3) circle(0.05cm);
\foreach\y in {-1,...,1}
\filldraw(-3, \y) circle(0.05cm);
\foreach\x in {-2,...,0}\foreach\y in {-2,...,0}
\filldraw(\x, \y) circle(0.05cm);
\foreach\y in {-1, ..., 0} 
\draw(-4, \y) circle (0.15cm); 
\foreach\x in {-1, ..., 0} 
\draw(\x, -4) circle (0.15cm); 
\draw(1, -3) circle (0.15cm); 
\draw(-3,1) circle (0.15cm); 
\node[right] at (1.1,-3){$\be$};
\draw (-3.3, 1.3)rectangle(1.5, -4.4); 
\node[below] at(1,1){$J_{\mathrm i}$};
\filldraw[gray, opacity=0.2] (-1.3,.3)rectangle(0.3,-4.3); 
\end{tikzpicture}
$$
\end{figure}
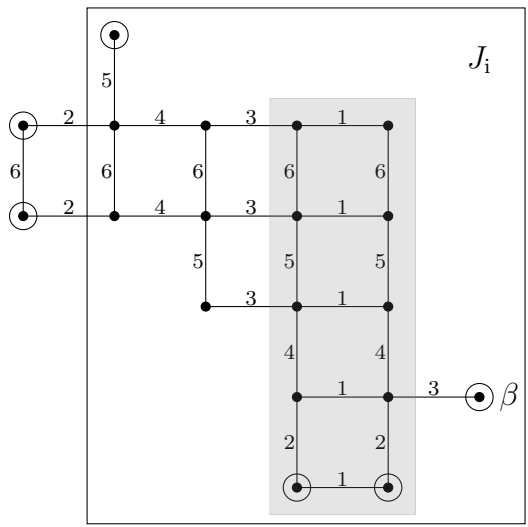

\begin{figure}[hpt]\caption
{$I\cong \mathrm E_{7, 7}$. The diagram represents  $(\be^\preq)$ for $\be=\al_{\{1,3,4,5,6,7,\}}$.  
The gray square covers the roots in $H=(\check\omega_7-\check\omega_2)^\perp$.} 

$$
\begin{tikzpicture}[scale=1.2]
\draw (-3, 0)--(-3, -1);
\draw (-3,0)--(0,0);
\draw (-4,-1)--(0,-1);
\draw (0,-3)--(0,0);
\draw (-1,-4)--(-1,0);
\draw(-2, 0)--(-2, -2)--(0, -2);
\draw (-1,-3)--(0,-3);
\foreach\x in {-3,...,0}\node[left] at (\x+.1,-0.5)  {$\ssty 6$};
\foreach\x in {-2,..., 0}\node[left] at (\x+.1,-1.5)  {$\ssty 5$};
\foreach\x in {-1,0}\node[left] at (\x+.1,-2.5)  {$\ssty 4$};
\foreach\x in {-1}\node[left] at (\x+.1,-3.5)  {$\ssty 2$};
\foreach\y in {-3,...,0}\node[above] at (-0.5,\y-.1)  {$\ssty 1$};
\foreach\y in {-2,..., 0}\node[above] at (-1.5,\y-.1)  {$\ssty 3$};
\foreach\y in {-1,0}\node[above] at (-2.5,\y-.1)  {$\ssty 4$};
\foreach\y in {-1}\node[above] at (-3.5,\y-.1)  {$\ssty 2$};
\filldraw(-4,-1) circle(0.05cm);
\filldraw(-1,-4) circle(0.05cm);
\foreach\x in {-1,...,0}
\filldraw(\x, -3) circle(0.05cm);
\foreach\y in {-1,...,0}
\filldraw(-3, \y) circle(0.05cm);
\foreach\x in {-2,...,0}\foreach\y in {-2,...,0}
\filldraw(\x, \y) circle(0.05cm);
\draw(-4, -1) circle (0.15cm); 
\draw(-1, -4) circle (0.15cm); 
\node[below] at (-1,-4.2){$\be$};
\filldraw[gray, opacity=0.2] (-3.2,.3)rectangle(0.3,-3.2); 
\end{tikzpicture}
$$
\end{figure}
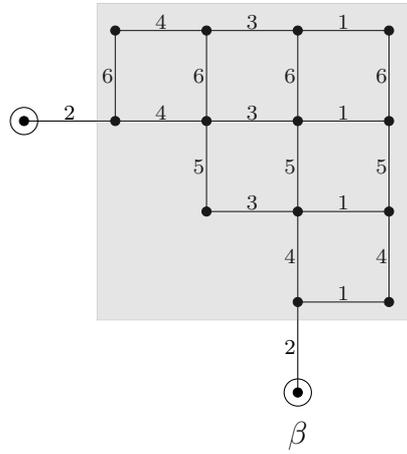

\section{Triangulations of standard parabolic facets}
In this section we prove Theorems \ref{main1} and \ref{main2}.

\par
Let $I$ be a face ideal of $\Phi^+$ and $F_I=\conv(I)$ the corresponding standard parabolic face. 
For all $J\subseteq I$, let
$$
\mc R_J=\{R\subseteq J\mid R \text{ reduced}\}.
$$
Then, let
$$
\mc T_I=\{\conv(R)\mid R\in \mc R_I, R \text{ maximal in } \mc R_I\}.
$$
We will prove that $\mc T_I$ is a triangulation of  $\mathrm F_I$. 
\par
By Propositions ~\ref{ab-nil-facet} and ~\ref{face-ab-nil}, it suffices to prove the 
claim when $I$ an abelian nilradical of $\Phi^+$.  
Henceforward, we make this assumption. 
\par

The proof is by induction on $\rk(\Phi)$ and is based on the existence of triangulation orders for all facet ideals. 
We start with two key lemmas. 

\par
For each $J\subseteq I$ let $\cone(J)$ be the positive cone generated by $J$, i.e. the set of linear combinations of elements in $J$ with nonnegative real coefficients.
Moreover, let $[J]$ be the {\it saturation} of $J$, i.e. 
$$[J]=\{x\in I\mid \exists\, y,z\in J\ y\leq x\leq z\}.$$ 

\begin{lem}\label{coni-bipartite}
Let $J$ be a saturated subset of $I$, and  $\{J_{\mathrm i},J_{\mathrm f}\}$ be a  bipartition of $J$.
Then $\cone(J)=\cone(J_{\mathrm i})\cup \cone(J_{\mathrm f})$. 
\end{lem}

\begin{proof}
The claim is obvious if the bipartition is not proper, in particular if $|J|\leq 2$. 
The inclusion $\cone(J_{\mathrm i})\cup \cone(J_{\mathrm f})\subseteq \cone(J)$ is obvious in all cases.
We prove the reverse inclusion by induction on $|J|$. 
It is immediate that, for any $K\subseteq J$, $\{K\cap J_{\mathrm i}, K\cap J_{\mathrm f}\}$ is a bipartition of $K$.  
Therefore, it suffices to prove that if the bipartition $\{J_{\mathrm i},J_{\mathrm f}\}$ of $J$ is proper,  there exists a proper saturated subset $K$ of $J$ such that $x\in \cone(K)$. 
\par

So, let $J_{\mathrm i}, J_{\mathrm f}\subsetneqq J$, $x\in \cone(J)$, and $x=\sum\limits_{\be\in J}c_\be\be$,  with $c_\be$ nonnegative real coefficients, be a fixed expression of $x$.
If $\{\be\in J\mid c_\be> 0\}$ is included in $J_{\mathrm i}$ or $J_{\mathrm f}$ we are done.
Also if the saturation $[\{\be\in J\mid c_\be> 0\}]$ is properly included in $J$ we are done.
Hence, we assume  $J=[\{\be\in J\mid c_\be> 0\}]$. 
This means that, for all $\be\in \Min\, J\cup  \Max\, J$, $c_\be> 0$.
Since $J_{\mathrm i}\smeno J_{\mathrm f}$ and $J_{\mathrm f}\smeno J_{\mathrm i}$ are an initial and a final section of $J$, we have  $\Min(J_{\mathrm i}\smeno J_{\mathrm f})\subseteq \Min\, J$ and $\Max(J_{\mathrm f}\smeno J_{\mathrm i})\subseteq \Max\, J$.
We fix  $\be_1\in \Min (J_{\mathrm i}\smeno J_{\mathrm f})$ and $\be_2\in \Max(J_{\mathrm f}\smeno J_{\mathrm i})$.
By Definition~\ref{bipartite}, and since $J$ is saturated, there exist $\ga_1, \ga_2\in J$  such that $\be_1+\be_2=\ga_1+\ga_2$ and $\be_1< \{\ga_1, \ga_2\}< \be_2$. 
Hence, $c_{\be_1}\be_1+ c_{\be_2} \be_2=(c_{\be_1}-c_{\be_2}) \be_1 + c_{\be_2} (\ga_1+\ga_2) =$ $(c_{\be_2}-c_{\be_1}) \be_2 + c_{\be_1} (\ga_1+\ga_2)$.
We obtain that, if  $c_{\be_1}\geq  c_{\be_2}$, then $x\in \cone (J \smeno\{\be_2\})$, while, if  $c_{\be_2}\geq  c_{\be_1}$, then $x\in \cone (J \smeno\{\be_1\})$. 
Since $\be_1$ and $\be_2$ are extremal elements in $J$, $J \smeno\{\be_1\}$ and $J \smeno\{\be_2\}$ are saturated, hence the claim is proved.
\end{proof}

\begin{lem}\label{coni-detachable}
Let $J$ be a saturated subset of $I$, $\be\in J$,  $\be$ detachable in  $J$, and $J_\be=$ \hbox{$\{\be\}\cup  (\red(\be)\cap J)$}. 
Then $\cone(J)=\cone(J_\be)\cup \cone(J\smeno\{\be\})$. 
\end{lem}

\begin{proof}
By Remark ~\ref{rem-detachable}, the claim is a direct consequence of Proposition ~\ref{coni-bipartite}.
\end{proof}

\begin{pro}\label{cover}
For all $J\subseteq I$, if $J$ is saturated, then
$$
\cone(J)=\bigcup\,\{\cone(R)\mid R\subseteq J, \  R \text{ reduced}\}.
$$
\end{pro}

\begin{proof}
The claim is obvious if $\rk(\Phi)=1$.
We assume $\rk(\Phi)\geq 2$ and  the claim true for any abelian nilradical in any irreducible root system of rank strictly lower than $\rk(\Phi)$.
\par
Let $J\subseteq I$ be saturated.
The inclusion \lq\lq$\supseteq$\rq\rq\ is clear, so it suffices to prove the reverse one.
Let $x\in \cone(J)$,  $\preq$ be a triangulation order on $I$, 
$\be_0=\max_\preq\{\be\in J\mid x\in \cone(J\cap(\be^\preq))\}$, and
$J_0=J\cap ({\be_0}^\preq)$.  
Then,  $x\in \cone(J_0)$ and $J_0$ is saturated, being the intersection of two saturated sets, hence it suffice to prove the claim for $J_0$. 
We rename $J:=J_0$, so that $\be_0=\min_\preq J$,  and $x\not\in \cone (J\smeno\{\be_0\})$.
\par

\nl(a) First, we consider the case $\be_0\in I\smeno S_{I,\preq}$.
Let $\Psi=\Phi\cap  \gen(I\smeno S_{I,\preq})$, $\Psi_1,\dots, \Psi_k$ be the irreducible components of $\Psi$, $I_i=I\cap \Psi_i$, $J_i=J\cap \Psi_i$.
Let \hbox{$\{c_\be\mid \be\in J_0\}$}  be a fixed set of nonnegative real coefficients such that \hbox{$x=\sum\limits_{\be\in  J_0} c_\be \be$} and 
let $x_i=\sum\limits_{\be\in J_i} c_{\be}\be$, for $i=1,\dots,k$.
Then,  $I_i$ is an abelian nilradical of $\Psi_i^+$ and $J_i$ is saturated in it, hence, by the induction assumption, there exists a subset $R_i$ of $J_i$, reduced relatively to $\Psi_i$, such that $x_i\in \cone(R_i)$,  for $i=1,\dots,k$. 
By Definition \ref{triang-order}, $I\cap \Psi$  is $\sim$closed and hence, by Lemma \ref{simclosed-components},  $R_1 \cup\dots\cup R_k$ is reduced in $\Phi$. 
Clearly,  $x\in \cone(R_1 \cup\dots\cup R_k)$, hence we are done.

\par
\nl(b) Then, we assume $\rk({\be_0}^\preq)= n$.
By definition, either $\be_0$ is detachable in  $({\be_0}^\preq)$, or $({\be_0}^\preq)$ has a bipartition $\{B_{\mathrm i}, B_{\mathrm f}\}$ such that $\be_0$ is a detachable element in both of $B_{\mathrm i}$ and $B_{\mathrm f}$. 
In this case, $\{J\cap B_{\mathrm i}, J\cap B_{\mathrm f}\}$ is a bipartition of $J$ and, by Lemma ~\ref{coni-bipartite}, we may fix a $B\in\{B_{\mathrm i}, B_{\mathrm f}\}$ such that  $x\in \cone(J\cap B)$.
If $\be_0$ is detachable in  $({\be_0}^\preq)$, we set $B=({\be_0}^\preq)$.
In any case, we define $J'=J\cap B$.
Then, we still have $\be_0=\min_\preq J'$, $x\in \cone(J')$ and, in any expression of $x$ as a linear combination of elements of $J'$, the coefficient of $\be_0$ is strictly positive. 
\par
Since $\be_0$ is detachable in $B$, there exists a  detaching  hyperplane $H$ for $\be_0$ in $B$, and it is clear that such an $H$ is a detaching hyperplane also for $\be_0$ in $J'$. 
By Remark~\ref{rem-detachable}, the two subsets \hbox{$\{\be_0\}\cup (J'\cap H)$} and \hbox{$J'\smeno \{\be_0\}$} form a bipartition of $J'$ and,
by Lemma~\ref{coni-bipartite}, we obtain $x\in \cone (\{\be_0\}\cup (J'\cap H))$. 
Hence, there exists a positive real coefficient $c_0$ such that  $x-c_0\be_0 \in \cone(J'\cap H)$. 
Now, $J'\cap H$ is contained in the abelian nilradical $I\cap H$ of  $(\Phi\cap  H)^+$.
Let $\Psi_1, \dots, \Psi_k$ be the irreducible components of $\Phi\cap  H$.
Arguing as in case (a), we find  $R_1, \dots, R_k$ such that  $R_i\subseteq J'\cap \Psi_i$, $R_i$ is reduced relatively to $\Psi_i$, and  $x-c_0\be_0 \in \cone(R_1\cup\dots\cup  R_k)$.
As before, $R_1\cup\dots\cup  R_k$ is reduced in $\Phi$, since 
$I\cap H$ is $\sim$closed.
Moreover, by definition of $\be_0$, $R_1\cup\dots\cup  R_k\subseteq ({\be_0}^\preq)$, hence in   $({\be_0}^\preq)\cap H$. 
By Definition \ref{triang-order}, this subset is contained in $\red(\be_0)$, hence 
$\{\be_0\}\cup R_1\cup\dots\cup R_k$ is reduced. 
This proves the claim. 
\end{proof}

\begin{rem}\label{cone-conv}
We observe that for each $J\subseteq I$, $\cone(J)\cap F_I=\conv(J)$. 
Indeed, if $x=\sum\limits_{\be\in J} c_\be \be$, and $\al_I$ is the simple root of $\Phi^+$ such that $I=(\al_I^\leq)$,  then $\sum\limits_{\be\in J} c_\be=(x, \check \omega_{\al_I})$, which is $1$ for all $x$ in $F_I$. 
\end{rem}

\begin{cor}\label{cor-cover}
Let $I$ be a facet ideal of $\Phi^+$ and 
$$
\mc T'_I=\{\conv(R)\mid R \in \mc R_I, \ \rk(R)=n\}.
$$
Then $\mc T'_I$ is a covering of $F_I$.
\end{cor}

\begin{proof}
By Proposition \ref{cover} and Remark \ref{cone-conv}, the set of all $\conv(R)$, with $R \subseteq I$ and $R$ reduced, is a covering of $F_I$. 
By standard topological arguments, we obtain that also $\mc T'_I$ is a covering  of $\mathrm F_I$.
\end{proof}

Our next step is to prove that the set $\mc T'_I$ defined in Corollary \ref{cor-cover} is a triangulation of the standard parabolic facet $\mathrm F_I$. 
For this, it remains  to prove that each   $T\in \mathcal T'_I$ is a simplex,  and that the intersection of any two $T_1, T_2\in \mathcal T_I$ is a common face of $T_1$ and $T_2$. 
This is proved in next two propositions.

\begin{pro}\label{reduced-independent}
Let $I$ be an abelian nilradical of $\Phi^+$ and $R$ be a reduced subset of $I$. Then $R$ is linearly independent. 
\end{pro}

\begin{proof}
We prove the claim by induction on $\rk(\Phi)$.
The case $\rk(\Phi)=1$ is obvious. We assume $\rk(\Phi)> 1$ and the claim true for irreducible root systems of rank lesser than $\rk(\Phi)$.
Let $\preq$ be a  triangulation order on $I$ and $\be=\min_\preq R$. 
\par
First we consider the case $\rk(\be^\preq)=n$ and $\be$ detachable in  $(\be^\preq)$. 
Let $H$ be a detaching hyperplane for $\be$ in  $(\be^\preq)$, $\Psi_1, \dots, \Psi_k$ be the irreducible components of  $\Phi\cap H$, and $R_i=(R\smeno\{\be\})\cap \Psi_i$, for $i=1, \dots, k$.
Then $R_i$ is contained in the abelian nilradical $I\cap \Psi_i$ of $\Psi_i^+$ and is reduced, relatively to $\Psi_i$.
By the induction assumption, it is linearly independent. 
Since $R\smeno\{\be\}=R_1\cup\dots\cup R_k$, we obtain that $R\smeno \{\be\}$, and hence $R$ are linearly independent.
\par
If $\rk(\be^\preq)=n$ and $\be$ is not detachable in  $(\be^\preq)$, there exists a bipartition  $\{J_{\mathrm i}, J_{\mathrm f}\}$ of  $(\be^\preq)$ such that $\be$ is a detachable element both in $J_{\mathrm i}$, and in $J_{\mathrm f}$.
By Definition ~\ref{bipartite}, either  $R\subseteq J_{\mathrm i}$, or $R\subseteq J_{\mathrm f}$, hence we can argue as in the previous case.
\par
If $\rk(\be^\preq)<n$, then $R$ is contained in the abelian nilradical $I\cap \gen(I\smeno S_{I, \preq})$, in $\Phi\cap \gen(I\smeno S_{I, \preq})$
and we may argue  by induction  as above. 
\end{proof}

\begin{pro}\label{intersection}
Let $I$ be an abelian nilradical of $\Phi^+$ and $R_1$, $R_2$ be reduced subsets in~$I$. 
Then, $\conv(R_1)\cap\conv(R_2)=\conv(R_1\cap R_2)$. 
In particular,  $\conv(R_1)\cap\conv(R_2)$ is a common face of  $\conv(R_1)$ and $\conv(R_2)$. 
\end{pro}

\begin{proof}
By Proposition ~\ref{reduced-independent}, $\conv(R_i)$ is the simplex with set of vertexes $R_i$, for $i=1,2$, hence  $\conv(R_1\cap R_2)$ is common face of  $\conv(R_1)$ and $\conv(R_2)$. 
Hence it suffices to prove the first statement.
The inclusion $\hbox{$\conv(R_1)\cap\conv(R_2)$}\supseteq \hbox{$\conv (R_1\cap R_2)$}$ is clear. 
We prove the reverse one, by induction on $\rk(\Phi)$. 

\par
If $\cone(R_1)\cap\cone(R_2)\subseteq \cone(R_1\cap R_2)$ then, by Remark \ref{cone-conv}, the analogous relation for the convex hulls hold. 
So we work with cones.

\par
For $\rk(\Phi)=1$ the claim is obvious. 
Let  $\rk(\Phi)>1$, $\preq$ be a fixed triangulation order on $I$,  and $\be=\min_\preq(R_1\cup R_2)$. 

\par
\nl (a) If $\rk(\be^\preq)< n$, then $R_1, R_2\subseteq I\smeno S_{I, \preq}$. 
Let $\Psi_1, \dots, \Psi_k$ be the connected components of $\Phi\cap \gen(I\smeno S_{I, \preq})$. 
Let $R_{j,i}=R_j\cap \Psi_i$, for $j=1,2$ and $i=1, \dots, k$. 
Each $R_{j,i}$ is a reduced subset in the abelian nilradical $I\cap \Psi_i$ of $\Psi_i^+$, hence by the induction assumption $\cone(R_{1,i})\cap\cone(R_{2,i})\subseteq \cone(R_{1,i}\cap R_{2,i})$, 
for each $i$ in $\{1,\dots,k\}$. 
This implies directly the inclusion $\cone(R_1)\cap\cone(R_2)\subseteq \cone(R_1\cap R_2)$. 

\nl
(b) Next, let  $\rk(\be^\preq)=n$,  $\be$ be detachable in  $(\be^\preq)$, $H$ be a detaching hyperplane, and $\ov R_i=$ $R_i\cap H$ for $i=1,2$.
Then it  is clear that  $R_i$ is contained in one of the two closed half spaces determined by $H$ in $\mathrm E$, hence it is easily seen that $\cone(R_i)\cap H=\cone(\ov R_i)$. 
Moreover, if $\be\in R_i$, then $R_i\smeno\{\be\}=\ov R_i$.
Now we distinguish two possibilities.
\nl
(b1) If $\be=\min_\preq R_i\prec \min_\preq R_{i'}$, with $\{i,i'\}=\{1,2\}$, then  $R_1$ and $R_2$ are weakly separated by $H$. Hence,  $R_1\cap R_2=\ov R_1\cap \ov R_2$ and,  moreover, $\cone(R_1)\cap\cone(R_2)=\cone(R_1)\cap H \cap \cone(R_2)=\cone(\ov R_1)\cap\cone(\ov R_2)$. 
Arguing as in case (a), with $\Phi\cap H$ in place of  $\Phi\cap \gen(I\smeno S_{I, \preq})$ and $\ov R_i$ in place of $R_i$, by the induction assumption we obtain $\cone(\ov R_1)\cap\cone(\ov R_2)\subseteq\cone (\ov R_1\cap \ov R_2)$, 
and hence the claim.

\nl
(b2) If $\be=\min_\preq R_1=\min_\preq R_2$, then for all $x\in \cone(R_1)\cap\cone(R_2)$ there exist $c_i\in \real$ and $\ov x_i\in \cone(\ov R_i)$ ($i=1, 2$) such that  $x=c_1\be+\ov x_1=c_2\be +\ov x_2$.
Since $\ov x_1, \ov x_2\in H$ and $\be\not\in H$, we must have  $c_1=c_2$ and hence $\ov x_1=\ov x_2$. 
It follows  $\ov x_1\in \cone (\ov R_1\cap \ov R_2)$ and  hence $x\in \cone (R_1\cap R_2)$.

\nl
(c) Finally, let $\rk(\be^\preq)=n$, $\be$ not be detachable in  $(\be^\preq)$, and  $\{J_{\mathrm i}, J_{\mathrm f}\}$  be a bipartition of $(\be^\preq)$. 
By definition, each of $R_1$ and $R_2$ is contained in exactly one of $J_{\mathrm i}$ and $J_{\mathrm f}$. 
If both are contained in $J_{\mathrm i}$, or both in $J_{\mathrm f}$, we are reduced to case (b). 
Otherwise, we may assume $R_1\subseteq J_{\mathrm i}$, $R_2\subseteq J_{\mathrm f}$, $R_1\cap  (J_{\mathrm i}\smeno J_{\mathrm f})\neq \emptyset$, and $R_2\cap  (J_{\mathrm f}\smeno J_{\mathrm i})\neq \emptyset$. 
Let $H$ be a separating hyperplane for the bipartition $\{J_{\mathrm i}, J_{\mathrm f}\}$, and $\ov R_i=R_i\cap H$, for $i=1,2$. 
For a fixed $i$ in $\{1,2\}$, $\conv(R_i)$ is contained in one of the half-spaces determined by $H$ in $\mathrm E$, hence  $H\cap\cone(R_i)=\cone (\ov R_i)$. 
Moreover,  $\cone(R_1)$ and $\cone(R_2)$  belong to opposite half-spaces with respect to $H$, hence $R_1\cap R_2=\ov R_1\cap \ov R_2$ and $\cone(R_1)\cap \cone(R_2)=\cone (\ov R_1)\cap\cone( \ov R_2)$. 
Arguing by induction as in case  (b1), we obtain $\cone (\ov R_1)\cap\cone( \ov R_2)\subseteq \cone (\ov R_1\cap \ov R_2)$, and hence the claim.
\end{proof}

Corollary \ref{cover} and Propositions \ref{reduced-independent} and \ref{intersection} imply directly the following theorem, which is Theorem \ref{main1}.

\begin{thm}
\label{main1-6}
Let $I$ be a facet ideal in $\Phi$ and
$$
\mc T'_I=\{\conv(R) \mid R\in \mc R_I, \rk(R)=n\}.
$$
Then $\mc T'_I$ is a triangulation of the facet ideal $\mathrm F_I$. 
\end{thm}

\begin{cor}
Each reduced subset in $I$ is a contained in a maximal reduced subset.
Moreover, each maximal reduced subset in $I$ has rank $n$, 
in particular is a linear basis of~$\mathrm E$. 
\end{cor}

\begin{proof}
Let $R_0$ be a reduced subset in $I$ such that $\rk (R_0)< n$. 
Let $x=\sum_{\be\in R_0} c_\be \be$ with $c_\be> 0$ for all $\be \in R_0$.
By Corollary ~\ref{cor-cover}, there exists a reduced subset $R$ in $I$ such that $\rk(R)=n$ and $x\in \conv(R)$. 
Then, by Proposition ~\ref{intersection}, $x\in \conv(R_0\cap R)$. 
It follows that $R_0\cap R=R_0$, hence $R_0$ is not maximal.
The rest of the claim follows from Proposition~\ref{reduced-independent}.
\end{proof}
 
We can finally prove the following result, which i is clearly equivalent to  Theorem \ref{main2}. 
The proof refers to the case by case analysis of Proposition \ref{triang-ord-exist}. 

\begin{thm}
\label{main2-6}
Let $I$ be a facet ideal in $\Phi$ and $R$ be a maximal reduced subset in $I$. 
Then $R$ is a $\ganz$-basis of the sub-lattice of $L(\Phi)$ generated by $(\Pi\smeno\{\al_I\})\cup \{m_{\al_I}\al_I\}$, where $\al_I$ is  the simple root such that $I=V_{\al_I}$.
\end{thm}

\begin{proof}
By Proposition \ref{face-ab-nil} and Remark \ref{rem-ab-nil}, 
it suffices to prove the claim in case $I$ is an abelian nilradical of $\Phi^+$, i.e. 
 $m_{\al_I}=1$. 
Under this assumption, we have to prove that $R$ is a $\ganz$-basis of~$L(\Phi)$.

\par
Let $\preq$ be a triangulation order of $I$ and $\be=\min_\preq R$. 
If $\be$ is detachable in  $(\be^\preq)$, let $J=(\be^\preq)\smeno\{\be\}$.
If $\be$ is not detachable in  $(\be^\preq)$, let $\{J_\mathrm i, J_\mathrm f\}$ be a bipartition of $(\be^\preq)$ such that $\be$ belongs to $J_\mathrm i$ and $J_\mathrm f$ and is detachable in  them.
In this case, $R$ is contained in exactly one of $J_\mathrm i$ and $J_\mathrm f$: we define $J=J_\mathrm i$ if $R\subseteq J_\mathrm i$, and $J=J_\mathrm f$ otherwise. 
In any case, there exists a hyperplane $H$ such that $\red(\be)\cap J=H\cap J$, hence $R\smeno\{\be\}$ is a reduced subset in the abelian nilradical $I\cap H$ of $(\Phi\cap H)^+$. 
Since $\rk(R\smeno\{\be\})=n-1$, also $\rk(I\cap H)=\rk(\Phi\cap H)=n-1$.
In particular $I\cap H$ has nontrivial intersection with each irreducible component of  $\Phi\cap H$. 
By Lemma \ref{induzione}, each of these intersections  is a nontrivial  abelian nilradical in its irreducible component, hence, by induction on the dimension, $R\smeno\{\be\}$ is a $\ganz$-basis of $L(\Phi\cap H)$.
\par
Now, we first consider the case in which $\be$ is long and is equal to $\min J$ or $\max J$ with respect to standard partial order.
In this case, as seen in the proof of Lemma~\ref{lem-detachable}, \hbox{$H=(\be^\vee-\check\omega_{\al_I})^\perp$}.
It follows directly that all simple roots different from $\al_I$ and perpendicular to $\be$ belong to $H$. 
For all other simple roots $\al\neq\al_I$, either $(\al, \be^\vee)=1$ and $\be-\al\in H$, or $(\al, \be^\vee)=-1$ and $\be+\al\in H$. 
Since the $\ganz$-span of $R\smeno\{\be\}$ contains all the roots in $H$, 
we obtain that the $\ganz$-span of $R$ contains $(\Pi\smeno\al_I)\cup \{\be\}$, and hence contains $\Pi$, as claimed.

\par
In the remaining cases, looking the proof of Proposition \ref{triang-ord-exist}, we can directly check that  for all $\al\in \Pi\smeno \{\al_I\}$, if $\be+\al\not\in\Phi$ then $\al\in H$ and, otherwise, $\be+\al\in \Phi\cap H$.
Arguing as in the previous case, we easily obtain that $R$ is a $\ganz$-basis of $L(\Phi)$.
\end{proof}

\section{Concluding remarks}

Via the action of the Weyl group, we may transport a triangulation of a standard parabolic facet to all facets in its orbit. 
Hence from the triangulations of all parabolic facets we obtain a triangulation  $\mc T$ of the whole boundary $\partial\pol$ of the root polytope $\pol$.
Clearly, such a $\mc T$ is not unique, since the way of transporting  a triangulation of a standard parabolic facet to the facets in its orbits is not unique; the possible ways  correspond to the systems of representatives of the left cosets of $W$ modulo the stabilizer of the standard parabolic facet.
For a fixed $\mc T$, for each $T\in  \mc T$, let $V_T$ be  the set of vertexes of $T$ and  $T_0=\conv(V_T\cup \{\0\})$. 
Then, clearly $\mc T_0:=\{T_0\mid T\in \mc T\}$ is a triangulation of $\pol$. 
Thus, the explicit enumeration of the maximal reduced subsets of  facet ideals, together with the above Theorem \ref{main2-6} and the results in \cite{CM1} would allow to compute the volume of $\pol$.
For the root types $\mathrm A$ and $\mathrm C$, this is done in \cite{CM2}. 
For the remaining types, it will be done in a next paper. 
In fact, the proof of Proposition \ref{triang-ord-exist} gives an explicit procedure for enumerating the reduced subsets.

\par
In  \cite{CM2}, with a suitable choice of the systems of representatives of the left cosets of $W$ modulo the stabilizers of the standard parabolic facets, we have obtained a triangulation of $\pol$  that restricts to a triangulation of the positive root polytope $\pol^+$, which by definition is $\conv(\Phi^+\cup \{\0\})$. 
In fact, this is a proof that, for the types  $\mathrm A$ and $\mathrm C$, the intersection of $\pol$ with the cone on $\Phi^+$ is equal to $\pol^+$.
This is one of the special properties of the root politope that hold only for the  types $\mathrm A$ and $\mathrm C$  (see also \cite{CM3}).
In fact, it is easy to see that, for all other root types, $\pol^+$ is properly contained  in $\pol\cap \cone(\Phi^+)$ \cite{Ch}.
Hence, in these cases, from the standard parabolic facets, we cannot obtain any triangulation of the positive root polytope.

\end{document}